\documentclass{amsart}

\usepackage[left=1.5in, right=1.5in]{geometry}
\usepackage{graphicx}
\usepackage{amsmath}
\usepackage{amsfonts}
\usepackage{amsthm}
\usepackage{enumitem}
\usepackage{amssymb}
\usepackage{url}
\usepackage{color}
\usepackage{mathrsfs}
\usepackage{mathtools}
\usepackage{mathabx}
\usepackage{dsfont}

\numberwithin{equation}{section}

\newtheorem{thm}{Theorem}[section]
\newtheorem{lem}[thm]{Lemma}
\newtheorem{prop}[thm]{Proposition}
\newtheorem{cor}[thm]{Corollary}

\theoremstyle{definition}
\theoremstyle{definition}
\theoremstyle{definition}\newtheorem*{remark}{Remark}

\newcommand{\norm}[1]{\left\lVert#1\right\rVert}
\newcommand{\RR}{\mathbb{R}}

\newcommand{\CC}{\mathbb{C}}
\newcommand{\ZZ}{\mathbb{Z}}
\newcommand{\NN}{\mathbb{N}}
\newcommand{\Intr}{\displaystyle\int}
\newcommand{\Sumn}{\displaystyle\sum}
\newcommand{\Suma}{\Sumn_{\alpha\in R_+}}
\newcommand*\diff{\mathop{}\!\mathrm{d}}

\newcommand{\al}{\alpha}
\newcommand{\alx}{\langle\alpha,x\rangle}

\newcommand{\DT}{\mathcal{D}_k}
\newcommand{\IntN}{\Intr_{\RR^N}}

\newcommand\restr[2]{{
  \left.\kern-\nulldelimiterspace  #1
  \vphantom{\big|}
  \right|_{#2}
  }}
\newcommand{\DSS}{H^1_\mathcal{D}(\RR^N)}
\newcommand{\DS}[1]{H^{#1}_\mathcal{D}}

\newcommand{\Rn}{\mathbb R^{N}}
\def\G{{\mathbb G}}

\begin{document}

 \title[Rellich, GN, Trudinger and CKN inequalities for Dunkl operators]
 {Rellich, Gagliardo-Nirenberg, Trudinger and Caffarelli-Kohn-Nirenberg inequalities for Dunkl operators and applications}

\author[A. Velicu]{Andrei Velicu}
\address{
  Andrei Velicu:
  \endgraf
Department of Mathematics
  \endgraf
  Imperial College London
  \endgraf
  180 Queen's Gate, London SW7 2AZ
  \endgraf
  United Kingdom
  \endgraf
  {\it E-mail address} {\rm a.velicu15@imperial.ac.uk}
  }

\author[N. Yessirkegenov]{Nurgissa Yessirkegenov}
\address{
  Nurgissa Yessirkegenov:
  \endgraf
  Institute of Mathematics and Mathematical Modelling
  \endgraf
  125 Pushkin str.
  \endgraf
  050010 Almaty
  \endgraf
  Kazakhstan
  \endgraf
  and
  \endgraf
  Department of Mathematics
  \endgraf
  Imperial College London
  \endgraf
  180 Queen's Gate, London SW7 2AZ
  \endgraf
  United Kingdom
  \endgraf
  {\it E-mail address} {\rm n.yessirkegenov15@imperial.ac.uk}
  }

\thanks{The second author was supported by the MESRK grant AP05133271. No new data was collected or generated
during
the course of research.}


\begin{abstract} In this paper we obtain weighted higher order Rellich, weighted Gagliardo-Nirenberg, Trudinger, Caffarelli-Kohn-Nirenberg inequalities and the uncertainty principle for Dunkl operators. Moreover, we introduce an extension of the classical Caffarelli-Kohn-Nirenberg inequalities. Furthermore, we give an application of Gagliardo-Nirenberg inequality to the Cauchy problem for the nonlinear damped wave equations for the Dunkl Laplacian.
\end{abstract}

\maketitle

\tableofcontents

\section{Introduction}
\label{SEC:intro}

The aim of this paper is to introduce an extension of the classical Caffarelli-Kohn-Nirenberg (CKN) inequalities with respect to ranges of parameters and to show analogues of the weighted higher order Rellich, Gagliardo-Nirenberg, Trudinger, CKN inequalities and the uncertainty principle for Dunkl operators. Moreover, we give an application of Gagliardo-Nirenberg inequality to the Cauchy problem for the nonlinear damped wave equations for the Dunkl Laplacian.

A systematic study of functional inequalities for Dunkl operators was started by the first author in \cite{Vel18} and \cite{Vel19}, which contain a thorough account of Sobolev and Hardy inequalities in the Dunkl setting. Apart from these main inequalities, the first paper contains a particular case of the Gagliardo-Nirenberg (in Corollary 5.7), while the second paper treated the basic case of the Rellich inequality with sharp constants (in Theorem 7.1) and a particular case of the Caffarelli-Kohn-Nirenberg inequality (in Theorem 7.2). After the proof of the latter Theorem we remarked that our result could be further generalised.

In this paper we fill in the gaps of the previous results by providing the details of the promised extension of the Caffarelli-Kohn-Nirenberg inequality, as well as related inequalities. In terms of the Rellich inequality, we generalise the previous results firstly by proving a weighted version, and also by considering higher powers of the Dunkl Laplacian. 

Similarly, for the Gagliardo-Nirenberg inequality we provide several extensions, both in the range of the coefficients used, and also by providing weighted results. We remark that similar results have been studied also by Mejjaoli: a special case of the Gagliardo-Nirenberg inequality for Dunkl operators was proved in \cite{Mejjaoli08}, and several other particular cases and related results also appear in \cite{Mejjaoli13} and \cite{Mejjaoli14}. However, our approaches for obtaining such Gagliardo-Nirenberg inequalities below are different than in these papers.

A remarkable consequence of our chosen method in the Caffarelli-Kohn-Nirenberg inequality is that in the particular case $k \equiv 0$, when Dunkl operators reduce to the usual partial derivatives, we obtain an extension of the classical result with respect to the range of parameters.

For the convenience of the reader let us now briefly recapture the main results of the paper. As usual, in this paper $A\lesssim B$ means that there exists a positive constant $c$ such that $A\leq c B$. If $A\lesssim B$ and $B\lesssim A$, then we write $A\approx B$. In these notations, if the left and right hand sides feature some functions $f$, the constant (using this notation) does not depend on $f$.

\subsection{Caffarelli-Kohn-Nirenberg inequalities}
\label{sub_SEC:CKN}

Recall the classical Caffarelli-Kohn-Nirenberg inequality \cite{CKN84}:

%%%%%%%%%%%%%%%%%%%%%%%%%%%%%%%%%%%%%%%%%%%%%%%%%%%%%%%%%%%%%%%%%%%%%%%%%%%%%%%%%%%%%%%%%%%%%%%%%
\begin{thm}\label{clas_CKN}
Let $N\in\mathbb{N}$ and let $p$, $q$, $r$, $a$, $b$, $d$, $\delta\in \mathbb{R}$ such that $p,q\geq1$, $r>0$, $0\leq\delta\leq1$, and
\begin{equation}\label{clas_CKN0}
\frac{1}{p}+\frac{a}{N},\, \frac{1}{q}+\frac{b}{N},\, \frac{1}{r}+\frac{c}{N}>0,
\end{equation}
where $c=\delta d + (1-\delta) b$. Then there exists a positive constant $C$ such that
\begin{equation}\label{clas_CKN1}
\||x|^{c}f\|_{L^r(\Rn)}\leq C \||x|^{a}|\nabla f|\|^{\delta}_{L^p(\Rn)} \||x|^{b}f\|^{1-\delta}_{L^q(\Rn)}
\end{equation}
holds for all $f\in C_c^{\infty}(\Rn)$, if and only if the following conditions hold:
\begin{equation}\label{clas_CKN2}
\frac{1}{r}+\frac{c}{N}=\delta \left(\frac{1}{p}+\frac{a-1}{N}\right)+(1-\delta)\left(\frac{1}{q}+\frac{b}{N}\right),
\end{equation}
\begin{equation}\label{clas_CKN3}
a-d\geq 0 \quad {\rm if} \quad \delta>0,
\end{equation}
\begin{equation}\label{clas_CKN4}
a-d\leq 1 \quad {\rm if} \quad \delta>0 \quad {\rm and} \quad \frac{1}{r}+\frac{c}{N}=\frac{1}{p}+\frac{a-1}{N}.
\end{equation}
\end{thm}
%%%%%%%%%%%%%%%%%%%%%%%%%%%%%%%%%%%%%%%%%%%%%%%%%%%%%%%%%%%%%%%%%%%%%%%%%%%%%%%%%%%%%%%%%%%%%%%%%

In this paper we extend the Caffarelli-Kohn-Nirenberg inequalities by enlarging the range of indices \eqref{clas_CKN0}. This will be achieved by using Dunkl theory methods. More precisely, we prove \eqref{clas_CKN1} with the usual gradient $\nabla f$ replaced by the Dunkl gradient $\nabla_k f$ (when $k=0$, $\nabla = \nabla_k$).

To illustrate our extension of Theorem \ref{clas_CKN}, let us state our Caffarelli-Kohn-Nirenberg inequalities:
\begin{itemize}[wide, labelwidth=!, labelindent=0pt]

\item ({\bf Caffarelli-Kohn-Nirenberg I inequality}) Let $1<p<N$, $1<q<\infty$, $0<r<\infty$ with $p+q\geq r$. Let also $\delta\in[0,1]\cap\left[\frac{r-q}{r},\frac{p}{r}\right]$ and $b$, $c\in\mathbb{R}$. Assume that $\frac{\delta r}{p}+\frac{(1-\delta)r}{q}=1$ and $ c=-\delta+b(1-\delta)$. Then there exists a positive constant $C$ such that
\begin{equation}\label{CKN1_ineq1_intro}
\||x|^{c}f\|_{L^r(\Rn)}\leq C\|\nabla f\|^{\delta}_{L^p(\Rn)}\||x|^{b}f\|^{1-\delta}_{L^q(\Rn)}
\end{equation}
holds for all $f\in C_c^{\infty}(\Rn)$.

\item ({\bf Caffarelli-Kohn-Nirenberg II inequality}) Let $q,r,a,b,c,\delta \in \RR$ such that $1<q<\infty$, $0<r<\infty$, $\frac{2N}{N-2}+q\geq r$, $\delta\in[0,1]\cap\left[\frac{r-q}{r},\frac{2N}{r(N-2)} \right]$ and $ N+2a-2>0$. Assume also that
$\frac{\delta r(N-2)}{2N}+\frac{(1-\delta)r}{q}=1$ and $c=\delta a+b(1-\delta)$. Then there exists a positive constant $C$ such that
\begin{equation}\label{CKN2_ineq1_intro}
\||x|^{c}f\|_{L^r(\Rn)}\leq C \||x|^{a}\nabla f\|^{\delta}_{L^2(\Rn)}\||x|^{b}f\|^{1-\delta}_{L^q(\Rn)}
\end{equation}
holds for all $f\in C_c^{\infty}(\Rn)$.

\item ({\bf Fractional Caffarelli-Kohn-Nirenberg inequality}) Let $a,b,c \in \RR$, $\delta\in[0,1]\cap\left[\frac{r-q}{r},\frac{2}{r} \right]$, $1<q<\infty$ and $0<r<\infty$ with $2+q\geq r$. Assume that $\frac{\delta r}{2}+\frac{(1-\delta)r}{q}=1$ and $c=\delta (a-1)+b(1-\delta)$, and also $1-N/2<a\leq 1$. Then we have the following fractional Caffarelli-Kohn-Nirenberg type inequalities for all $f\in \mathcal{S}(\Rn)$:
\begin{equation}\label{CKN3_ineq1_intro}
\||x|^{c}f\|_{L^r(\Rn)}\leq\frac{1}{C(a)^{\delta}} \|(-\Delta)^{(1-a)/2}f\|^{\delta}_{L^2(\Rn)}\||x|^{b}f\|^{1-\delta}_{L^q(\Rn)},
\end{equation}
where the constant $C(a)$ (with $\gamma=0$) is given in Theorem \ref{CKN3_thm}.
\end{itemize}
%%%%%%%%%%%%%%%%%%%%%%%%%%%%%%%%%%%%%%%%%%%%%%%%%%%%%%%%%%%%%%%%%%%%%%%%%%%%%%%%%%%%%%%%%%%%%%%%%

%%%%%%%%%%%%%%%%%%%%%%%%%%%%%%%%%%%%%%%%%%%%%%%%%%%%%%%%%%%%%%%%%%%%%%%%%%%%%%%%%%%%%%%%%%%%%%%%%
For example, in the special case $1<p=q=r<N$, $b=-\frac{N}{p}$ and $c=-\frac{\delta (N-p)-N}{p}$, we get from \eqref{CKN1_ineq1_intro} for all $f\in C_{0}^{\infty}(\mathbb{R}^{N})$ that
\begin{equation}\label{CKN_example1}
\left\|\frac{f}{|x|^{\frac{\delta (N-p)-N}{p}}}\right\|_{L^p(\Rn)}\lesssim
\|\nabla f\|^{\delta}_{L^p(\Rn)}
\left\|\frac{f}{|x|^{\frac{N}{p}}}\right\|^{1-\delta}_{L^p(\Rn)}, \quad 0\leq\delta\leq1.
\end{equation}
Since we have
$$\frac{1}{q}+\frac{b}{N}=\frac{1}{p}+\frac{1}{N}\left(-\frac{N}{p}\right)=0,$$
we see that \eqref{clas_CKN0} fails, so that the inequality \eqref{CKN_example1} is not covered by Theorem \ref{clas_CKN}.

%%%%%%%%%%%%%%%%%%%%%%%%%%%%%%%%%%%%%%%%%%%%%%%%%%%%%%%%%%%%%%%%%%%%%%%%%%%%%%%%%%%%%%%%%%%%%%%%%

\begin{remark} For other extended Caffarelli-Kohn-Nirenberg inequalities, we refer to \cite{RSY17_ext_fr} and \cite{RSY18_ext_trans} for $f\in C_{0}^{\infty}(\G\setminus\{0\})$ on general homogeneous group, to \cite{RSY17_strat} on stratified groups, to \cite{RY18_hypo} for general hypoelliptic differential operators, to \cite{RY18_gen_Lie} on general Lie group, to \cite{RY18_Riem} on Riemannian manifolds with negative curvature and references therein.
\end{remark}

In Section \ref{SEC:CKN} we show \eqref{CKN1_ineq1_intro}-\eqref{CKN3_ineq1_intro} as
a consequence of general results for Dunkl operators.

\subsection{Gagliardo-Nirenberg, Trudinger and Rellich inequalities}
\label{sub_SEC:GN}

The Gagliardo-Nirenberg inequality was proved independently by Gagliardo \cite{Gagliardo} and Nirenberg \cite{Nirenberg} and it states that
\begin{equation} \label{classicalGN}
\norm{\nabla^j f}_{L^q(\Rn)} \leq C \norm{\nabla^l f}_{L^r(\Rn)}^\theta \norm{f}_{L^p(\Rn)}^{1-\theta},
\end{equation}
for all $f\in W^{l,r}(\RR^N)\cap L^p(\RR^N)$, where $1\leq p \leq \infty$, $j,l\in\NN$ and $\frac{j}{l} \leq \theta \leq 1$ are such that
$$ \frac{1}{q} = \frac{j}{N} + \theta \left( \frac{1}{r} - \frac{l}{N} \right) + \frac{1-\theta}{p}.$$
There is an exceptional case when $r>1$ and $l-j-\frac{N}{r}$ is a non-negative integer, in which case we must also require $\theta <1$. See \cite{FFRS} for historical remarks on this result and a complete proof, and also the references therein for an account of the many applications of the Gagliardo-Nirenberg inequality to PDEs.

Assuming the notation specific to Dunkl theory, which we introduce in section \ref{SEC:prelim}, our results for Dunkl operators are the following:

\begin{itemize}[wide, labelwidth=!, labelindent=0pt]

\item ({\bf Weighted higher order Rellich inequality}) Let $j\in \mathbb{N}$, $a+2(j-1)+1\leq b\leq a+2(j-1)+2$ and $p=\frac{2(N+2\gamma)}{N+2\gamma-2+2(b-(a+2(j-1)+1))}$, and suppose that $a+2(j-1)+1<\frac{N+2\gamma-2}{2}$. Then, for any $f\in C_{c}^{\infty}(\Rn)$ we have the inequality
\begin{equation}\label{high_Rellich_ineq1_intro}
\left\|\frac{f}{|x|^{b}}\right\|_{p}\lesssim \left\|\frac{\Delta_{k}^{j}f}{|x|^{a}}\right\|_{2}.
\end{equation}

\item ({\bf Uncertainty type principle}) Let $1<p<\frac{N+2\gamma}{1+2\gamma}$ and $\frac{1}{p}+\frac{1}{q}=1$. Then, for any $f\in C_c^{\infty}(\Rn)$, we have the inequality
\begin{equation}\label{uncer_1_intro}
\left(\int_{\Rn}|\nabla_{k}f|^{p}\diff\mu_{k}\right)^{\frac{1}{p}}\left(\int_{\Rn}|x|^{q}|f|^{q} \diff\mu_{k}\right)^{\frac{1}{q}} \gtrsim \int_{\Rn}|f|^{2}\diff\mu_{k}.
\end{equation}

\item ({\bf Gagliardo-Nirenberg inequality I}) Let $1\leq p,q < \infty$, $1\leq r <N+2\gamma$, and $0\leq \theta \leq 1$ satisfying the assumption $\theta \left(\frac{1}{N+2\gamma}  + \frac{1}{p}-\frac{1}{r}\right) =\frac{1}{p}-\frac{1}{q}$. Then the following inequality holds for all $f\in \DSS \cap L^p(\mu_k)$
\begin{equation} \label{GN0_ineq1_intro}
\norm{f}_{q} \lesssim \norm{\nabla_k f}_r^\theta \norm{f}_p^{1-\theta}.
\end{equation}

\item ({\bf Gagliardo-Nirenberg inequality II}) Let $1<p<\infty$, $0 \leq s \leq \frac{N+2\gamma}{p}$ and $0\leq \theta\leq 1$. Then, for any $f \in \mathcal{S}(\RR^N)$, we have the inequality
\begin{equation} \label{GN1_ineq1_intro}
 \norm{(-\Delta_k)^{s(1-\theta)/2}f}_p
	\lesssim \norm{(-\Delta_k)^{s/2}f}_p^{1-\theta} \norm{f}_p^\theta.
\end{equation}

\item ({\bf Weighted Gagliardo-Nirenberg I inequality}) Let $1<p<\frac{N+2\gamma}{2+2\gamma}$, $2\leq s\leq \frac{N+2\gamma}{p}$, and define $q=\frac{p(N+2\gamma)}{N+2\gamma-p}$. Then, for any $f\in C_c^\infty(\RR^N)$
we have
\begin{equation}\label{GN2_ineq1_intro}
\left\|\frac{f}{|x|}\right\|_{q}\lesssim \|(-\Delta_{k})^{s/2}f\|_{p}^{2/s}\|f\|^{1-2/s}_{p}.
\end{equation}

\item ({\bf Weighted Gagliardo-Nirenberg II inequality}) Assume that $N+2\gamma>2$. Let $1\leq a\leq 2$, $2\leq s\leq \frac{N+2\gamma}{2}$, and define $p=\frac{2(N+2\gamma)}{N+2\gamma-2+2(a-1)}$. Then, for any $f\in C_c^\infty(\RR^N \setminus \{0\})$
we have
\begin{equation}\label{GN3_ineq1_intro}
\left\|\frac{f}{|x|^{a}}\right\|_{p}\lesssim \|(-\Delta_{k})^{s/2}f\|_{2}^{2/s}\|f\|^{1-2/s}_{2}.
\end{equation}

\item ({\bf Weighted Gagliardo-Nirenberg III inequality}) Let $0\leq a\leq s \leq \frac{N+2\gamma}{2}$. Then, for any $f\in \mathcal{S}(\Rn)$
we have
\begin{equation}\label{GN4_ineq1_intro}
\left\|\frac{f}{|x|^{a}}\right\|_{2}\lesssim \|(-\Delta_{k})^{s/2}f\|_{2}^{a/s}\|f\|^{1-a/s}_{2}.
\end{equation}

\item ({\bf Trudinger inequality})
Let $1<p<\infty$, and $\frac{1}{p}+\frac{1}{p'}=1$. Then, there exist constants $a,C>0$ such that for any $f\in\mathcal{S}(\RR^N)$ that satisfies the assumption
\begin{equation} \label{trudinger_assumption_intro}
\norm{(-\Delta_k)^{(N+2\gamma)/(2p)} f}_p \leq 1,
\end{equation}
we have the inequality
\begin{equation} \label{trudinger_ineq_intro}
\int_{\RR^N} \left( \exp(a |f(x)|^{p'}) - \sum_{\substack{ 0\leq j <p-1 \\ j\in\NN}} \frac{1}{j!} (a |f(x)|^{p'})^j \right) \diff\mu_k(x) 
\leq C \norm{f}_p^p.
\end{equation}

\end{itemize}
In the special case when $\gamma=0$, hence $\Delta_{k}=\Delta$, the usual gradient in $\RR^N$, the inequality \eqref{trudinger_ineq_intro} was investigated in \cite{ogawa1990proof} for $p=2$, in \cite{ogawa1991trudinger} for $p=N=2$, in \cite{adachi2000trudinger} for $p=N\geq2$ and in \cite{Ozawa95} for $p$ and $N$ as \eqref{trudinger_ineq_intro}. We also refer to \cite{RY17_Trudinger} and \cite{RY18_hypo} for hypoelliptic versions of Trudinger inequalities, and to \cite{RY18_Riem} as well as to references therein on Riemannian manifolds with negative curvature.

\subsection{An Application}

Finally, using the Gagliardo-Nirenberg inequality we study the Cauchy problem for the following nonlinear damped wave equations for the Dunkl Laplacian
\begin{equation} \label{Cauchy_problem}
\begin{cases}
\partial_{tt}^2 u(t,x) - \Delta_k u(t,x) + b \partial_t u(t,x) + m u(t,x) = f(u)(t,x) \\
u(0,x)=\epsilon u_1(x) \\
\partial_t u(0,x) = \epsilon u_2(x),
\end{cases}
\end{equation}
with nonlinearity satisfying the assumptions $f(0)=0$ and
$$ |f(a)-f(b)|\leq C (|a|^{p-1} + |b|^{p-1}) |a-b|,$$
for some $p>1$. We first study the linear version of this equation (corresponding to $f=0$) using techniques based on the Dunkl transform, and then we prove a global existence and uniqueness result for the general Cauchy problem \eqref{Cauchy_problem} with small data.

The paper is organised as follows. In Section \ref{SEC:prelim} we briefly introduce Dunkl operators and their important properties. Rellich, Gagliardo-Nirenberg and Caffarelli-Kohn-Nirenberg inequalities are then discussed in Sections \ref{SEC:Rellich}, \ref{SEC:GN} and \ref{SEC:CKN}, respectively. Finally, in Section \ref{SEC:Appl} we discuss an application to a nonlinear damped wave equation.

\section{Preliminaries}
\label{SEC:prelim}

We present in this section a very quick introduction to Dunkl operators; see the excellent survey papers \cite{Rosler} and \cite{Anker} for more details.

In this paper, a root system is a finite set $R\subset \RR^N\setminus \{0\}$ that satisfies $R \cap \alpha \RR = \{ -\alpha, \alpha\}$ and $\sigma_\alpha(R) = R$ for all $\alpha\in R$. Here $\sigma_\alpha$ is the reflection in the hyperplane orthogonal to the root $\alpha$, i.e.,
$$ \sigma_\alpha x = x - 2 \frac{\alx}{\langle \alpha,\alpha \rangle} \alpha.$$
The reflection group generated by all $\sigma_\alpha$ for $\alpha\in R$ is finite, and we call it $G$.

A multiplicity function is a map $k:R \to [0,\infty)$, denoted $\alpha \mapsto k_\alpha$, which is $G$-invariant, i.e., $k_\alpha=k_{g\alpha}$ for all $g\in G$ and all $\alpha\in R$. Any root system $R$ can be written as a disjoint union $R=R_+ \cup (-R_+)$, and we call $R_+$ a positive subsystem. Of course there are in general many choices for the positive subsystem $R_+$, but once fixed, this choice will not make a difference in the definition below due to the $G$-invariance of the multiplicity function.

From now on we fix a root system $R$ in $\RR^N$ with positive subsystem $R_+$ and multiplicity function $f$. For simplicity, we also assume without loss of generality that $|\alpha|^2=2$ for all roots $\al \in R$. Dunkl operators are defined on $C^1(\RR^N)$ by
$$ T_i f(x) = \partial_i f(x) + \Suma k_\alpha \alpha_i \frac{f(x)-f(\sigma_\alpha x)}{\alx},$$
for each $i=1,\ldots, N$. The Dunkl gradient and Dunkl laplacian are defined as $\nabla_k=(T_1,\ldots, T_N)$ and $\Delta_k = \displaystyle\sum_{i=1}^N T_i^2$, respectively. Note that in the special case $k=0$ Dunkl operators reduce to partial derivatives, and $\nabla_0=\nabla$ and $\Delta_0=\Delta$ are the usual gradient and laplacian.

%We can express the Dunkl laplacian in terms of the usual gradient and laplacian using the following formula:
%%
%\begin{equation} \label{Dunkllaplacian}
%\Delta_k f(x) = \Delta f(x) + 2\Suma k_\alpha \left[ \frac{\langle \nabla f(x),\alpha \rangle}{\alx} - \frac{f(x)-f(\sigma_\alpha x)}{\alx^2} \right].
%\end{equation}
%%

The natural spaces to study Dunkl operators are the weighted spaces $L^p(\mu_k)$, where $\diff\mu_k =w_k(x) \diff x$, and the weight function is given by
$$ w_k(x) = \prod_{\alpha\in R_+} |\alx|^{2k_\alpha}.$$
This is a homogeneous function of degree
$$ \gamma := \Suma k_\alpha.$$
The measure of the weighted spaces $L^p(\mu_k)$ will be written simply $\norm{\cdot}_p$. With respect to this weighted measure we have the integration by parts formula
$$ \IntN T_i(f) g \diff\mu_k = - \IntN f T_i(g) \diff\mu_k.$$
%

%If one of the functions $f,g$ is $G$-invariant, then we have the Leibniz rule
%%
%$$ T_i(fg) = f T_ig + g T_if.$$
%%
%In general we have
%%
%$$ T_i(fg)(x) = T_if(x)g(x) + f(x)T_ig(x) - \Suma k_\alpha \alpha_i \frac{(f(x)-f(\sigma_\alpha x))(g(x)-g(\sigma_\alpha x))}{\alx}.$$
%%

A Sobolev inequality is available for the Dunkl gradient (see \cite{Vel18}).

\begin{thm}  \label{sobolev}
Let $1 \leq p < N+2\gamma$ and $q=\frac{p(N+2\gamma)}{N+2\gamma-p}$. Then there exists a constant $C>0$ such that we have the inequality
$$ \norm{f}_q \leq C \norm{\nabla_k f}_p \qquad \forall f\in C_c^\infty(\RR^N).$$
\end{thm}

The Dunkl kernel $E_k(x,y)$, defined on $\CC^N\times\CC^N$, acts as a generalisation of the usual exponential $e^{\langle x,y\rangle}$, and is defined, for fixed $y\in\CC^N$, as the unique solution $Y=E_k(\cdot, y)$ of the system of equations
$$ T_iY=y_i Y, \qquad i=1,\ldots N,$$
which is real analytic on $\RR^N$ and satisfies $Y(0)=1$. Another way to define the Dunkl exponential is in terms of the intertwining operator $V_k$ which connects the usual partial derivatives to Dunkl operators via the relation
$$ T_i V_k = V_k \partial_i.$$
The Dunkl exponential can then be equivalently defined as
$$ E_k(x,y) = V_k \left( e^{\langle \cdot, y \rangle} \right) (x).$$

It is in general difficult to characterise the Dunkl kernel $E_k$, but we have the bounds
$$ |\partial_y^\beta E_k(x,y)| \leq |x|^{|\beta|} \displaystyle\max_{g\in G} e^{\text{Re}\langle gx,y\rangle},$$
which hold for all $x\in\RR^N$, $y\in\CC^N$, and all $\beta\in\ZZ^N_+$.

With the help of these bounds, we can define the Dunkl transform on $L^1(\mu_k)$ by
$$ \mathcal{D}_k(f)(\xi)= \frac{1}{M_k}\Intr_{\RR^N} f(x)E_k(-i\xi,x)\diff \mu_k(x), \qquad \text{ for all } \xi\in\RR^N,$$
where
$$ M_k = \IntN e^{-|x|^2/2} \diff\mu_k(x)$$
is the Macdonald-Mehta integral. The Dunkl transform extends to an isometric isomorphism of $L^2(\mu_k)$; in particular, the Plancherel formula holds. When $k=0$ the Dunkl transform reduces to the Fourier transform.

The Sobolev spaces associated to Dunkl operators are the spaces $\DS{s}$ defined for all $s>0$ and which consist of all functions $f\in L^2(\mu_k)$ for which the map $\xi \mapsto |\xi|^s \DT(f)(\xi)$ belongs to $L^2(\mu_k)$. Then $\DS{s}$ is a Banach space with norm
$$ \norm{f}_{\DS{s}}^2 := \IntN |\DT(f)(\xi)|^2(1+|\xi|^2)^s \diff\mu_k(\xi).$$

\section{Weighted higher order Rellich inequalities}
\label{SEC:Rellich}
In this section we prove the uncertainty type principle and the Rellich inequalities.

The classical Rellich inequality for Dunkl operators was studied in \cite{Vel19} (see also the references therein) where the following result was proved:

\begin{thm} \label{classicalRellich}
Suppose that $N+2\gamma \neq 2$. Then, for any $f\in C_c^\infty(\RR^N \setminus \{0\})$, we have the inequality
$$ \norm{\Delta_k f}_2^2 \geq \frac{(N+2\gamma)^{2}(N+2\gamma-4)^2}{16} \norm{\frac{f}{|x|^2}}_2^2.$$
The constant in this inequality is sharp.
\end{thm}

In this section we will generalise this result firstly by adding a weight, and secondly by considering higher powers of the Laplacian. The first result is the following:

%%%%%%%%%%%%%%%%%%%%%%%%%%%%%%%%%%%%%%%%%%%%%%%%%%%%%%%%%%%%%%%%%%%%%%%%%%%%%%%%%%%%%%%%%%%%%%%%%
\begin{thm}\label{weigh_Rellich_thm}
Let $a+1\leq b\leq a+2$ and $p=\frac{2(N+2\gamma)}{N+2\gamma-2+2(b-(a+1))}$, and suppose that $a+1<\frac{N + 2\gamma-2}{2}$. Then, for any $f\in C_c^{\infty}(\Rn)$ we have the inequality
\begin{equation}\label{weigh_Rellich_ineq1}
\left\|\frac{f}{|x|^{b}}\right\|_{p}\lesssim \left\|\frac{\Delta_{k}f}{|x|^{a}}\right\|_{2}.
\end{equation}
\end{thm}
%%%%%%%%%%%%%%%%%%%%%%%%%%%%%%%%%%%%%%%%%%%%%%%%%%%%%%%%%%%%%%%%%%%%%%%%%%%%%%%%%%%%%%%%%%%%%%%%%

\begin{remark} In the special case $a=0$ and $b=p=2$, this result gives Theorem \ref{classicalRellich}, and this result, of course, implies classical Rellich inequalities when $\gamma = 0$, hence $\Delta_{k}=\Delta$ and $\diff\mu_{k}=\diff x$.
\end{remark}

\begin{remark} Similar inequalities to \eqref{weigh_Rellich_ineq1} for sub-Laplacians with
drift on stratified groups were investigated in \cite{RY19_drift}.
\end{remark}

Before we prove this inequality, we need the following Caffarelli-Kohn-Nirenberg type result from \cite{Vel19}.

%%%%%%%%%%%%%%%%%%%%%%%%%%%%%%%%%%%%%%%%%%%%%%%%%%%%%%%%%%%%%%%%%%%%%%%%%%%%%%%%%%%%%%%%%%%%%%%%%
\begin{thm}\label{weigh_Hardy_thm}
Let $a\leq b\leq a+1$ and $p=\frac{2(N+2\gamma)}{N+2\gamma-2+2(b-a)}$, and suppose that $a<\frac{N+2\gamma-2}{2}$. Then, for any $f\in C_{c}^{\infty}(\Rn)$ we have the inequality
\begin{equation}\label{weigh_Hardy_ineq1}
\left\|\frac{f}{|x|^{b}}\right\|_{p}\lesssim \left\|\frac{\nabla_{k}f}{|x|^{a}}\right\|_{2}.
\end{equation}
\end{thm}
%%%%%%%%%%%%%%%%%%%%%%%%%%%%%%%%%%%%%%%%%%%%%%%%%%%%%%%%%%%%%%%%%%%%%%%%%%%%%%%%%%%%%%%%%%%%%%%%%

\begin{proof}[Proof of Theorem \ref{weigh_Rellich_thm}] Note that the numbers $a+1$, $b$ and $p$ satisfy the assumptions of Theorem \ref{weigh_Hardy_thm}, so we have from \eqref{weigh_Hardy_ineq1} that
\begin{equation}\label{weigh_Rellich_ineq1_1}
\left\|\frac{f}{|x|^{b}}\right\|_{p}\lesssim \left\|\frac{\nabla_{k}f}{|x|^{a+1}}\right\|_{2}.
\end{equation}
On the other hand, since $a+1<\frac{N+2\gamma-2}{2}$, hence $a<\frac{N+2\gamma-2}{2}$, by \eqref{weigh_Hardy_ineq1} again we have
\begin{equation}\label{weigh_Rellich_ineq1_2}
\left\|\frac{\nabla_{k}f}{|x|^{a+1}}\right\|_{2}\lesssim \left\|\frac{\Delta_{k}f}{|x|^{a}}\right\|_{2}.
\end{equation}
Thus, combining \eqref{weigh_Rellich_ineq1_1} and \eqref{weigh_Rellich_ineq1_2} we obtain \eqref{weigh_Rellich_ineq1}.
\end{proof}

Now let us introduce weighted higher order Rellich inequality for Dunkl operators.

%%%%%%%%%%%%%%%%%%%%%%%%%%%%%%%%%%%%%%%%%%%%%%%%%%%%%%%%%%%%%%%%%%%%%%%%%%%%%%%%%%%%%%%%%%%%%%%%%
\begin{thm}\label{high_Rellich_thm}
Let $j\in \mathbb{N}$, $a+2(j-1)+1\leq b\leq a+2(j-1)+2$ and
$$p=\frac{2(N+2\gamma)}{N+2\gamma-2+2(b-(a+2(j-1)+1))},$$
and suppose that $a+2(j-1)+1<\frac{N+2\gamma-2}{2}$. Then, for any $f\in C_c^{\infty}(\Rn)$ we have the inequality
\begin{equation}\label{high_Rellich_ineq1}
\left\|\frac{f}{|x|^{b}}\right\|_{p}\lesssim \left\|\frac{\Delta_{k}^{j}f}{|x|^{a}}\right\|_{2}.
\end{equation}
\end{thm}
%%%%%%%%%%%%%%%%%%%%%%%%%%%%%%%%%%%%%%%%%%%%%%%%%%%%%%%%%%%%%%%%%%%%%%%%%%%%%%%%%%%%%%%%%%%%%%%%%

\begin{proof}[Proof of Theorem \ref{high_Rellich_thm}] We prove this theorem by the induction method. When $j=1$ the inequality \eqref{high_Rellich_ineq1} holds true since this is the case of Theorem \ref{weigh_Rellich_thm}. Now, assume this holds for the case $j=m$. Namely, for $m\in \mathbb{N}$, $a+2(m-1)+1\leq b\leq a+2(m-1)+2$, $p=\frac{2(N+2\gamma)}{N+2\gamma-2+2(b-(a+2(m-1)+1))}$ and $a+2(m-1)+1<\frac{N+2\gamma-2}{2}$, the inequality
\begin{equation}\label{high_Rellich_ineq2}
\left\|\frac{f}{|x|^{b}}\right\|_{p}\lesssim \left\|\frac{\Delta_{k}^{m}f}{|x|^{a}}\right\|_{2}
\end{equation}
holds for any $f\in C_c^{\infty}(\Rn)$. So, using this we need to prove the following: For $m\in \mathbb{N}$, $a+2m+1\leq b\leq a+2m+2$, $p=\frac{2(N+2\gamma)}{N+2\gamma-2+2(b-(a+2m+1))}$ and $a+2m+1<\frac{N+2\gamma-2}{2}$, the inequality
\begin{equation}\label{high_Rellich_ineq3}
\left\|\frac{f}{|x|^{b}}\right\|_{p}\lesssim \left\|\frac{\Delta_{k}^{m+1}f}{|x|^{a}}\right\|_{2}
\end{equation}
holds for any $f\in C_{c}^{\infty}(\Rn)$. Replacing $a$ by $a+2$ in \eqref{high_Rellich_ineq2} we get
\begin{equation}\label{high_Rellich_ineq4}
\left\|\frac{f}{|x|^{b}}\right\|_{p}\lesssim \left\|\frac{\Delta_{k}^{m}f}{|x|^{a+2}}\right\|_{2}
\end{equation}
for $a+2m+1\leq b\leq a+2m+2$, $p=\frac{2(N+2\gamma)}{N+2\gamma-2+2(b-(a+2m+1))}$ and $a+2m+1<\frac{N+2\gamma-2}{2}$.

On the other hand, since $a+2m+1<\frac{N+2\gamma-2}{2}$, hence $a+1<\frac{N+2\gamma-2}{2}$, replacing $f$ by $\Delta_{k}^{m}f$ in \eqref{weigh_Rellich_ineq1} one has
\begin{equation}\label{high_Rellich_ineq5}
\left\|\frac{\Delta_{k}^{m}f}{|x|^{a+2}}\right\|_{2}\lesssim \left\|\frac{\Delta_{k}^{m+1}f}{|x|^{a}}\right\|_{2}.
\end{equation}
Thus, \eqref{high_Rellich_ineq4} and \eqref{high_Rellich_ineq5} imply \eqref{high_Rellich_ineq3}.

The proof is complete.
\end{proof}
Now let us give the following uncertainty principle for the Dunkl gradient:
\begin{thm}\label{uncer_thm}
Let $1<p<\frac{N+2\gamma}{1+2\gamma}$ and $\frac{1}{p}+\frac{1}{q}=1$. Then, for any $f\in C_c^{\infty}(\Rn)$, we have the inequality
\begin{equation}\label{uncer_1}
\left(\int_{\Rn}|\nabla_{k}f|^{p}\diff\mu_{k}\right)^{\frac{1}{p}}\left(\int_{\Rn}|x|^{q}|f|^{q}\diff\mu_{k}\right)^{\frac{1}{q}} \gtrsim \int_{\Rn}|f|^{2}\diff\mu_{k}.
\end{equation}
\end{thm}
\begin{remark}
In the special case when $\gamma=0$ we have $\nabla_{k}=\nabla$, the usual gradient in $\Rn$, then \eqref{uncer_1} with $p=q=2$ and $N\geq 3$ implies the classical uncertainty principle
$$\left(\int_{\Rn}|\nabla f|^{2}\diff x\right)^{\frac{1}{2}}\left(\int_{\Rn}|x|^{2}|f|^{2}\diff x\right)^{\frac{1}{2}} \gtrsim \int_{\Rn}|f|^{2}\diff x$$
for $f\in C_c^{\infty}(\Rn)$.
\end{remark}

Before proving Theorem \ref{uncer_thm}, let us recall the following $L^p$ Hardy inequality from \cite{Vel19}

\begin{thm}\label{Lp_Hardy_thm}
Let $1<p<\frac{N+2\gamma}{1+2\gamma}$. Then, for any $f\in C_c^{\infty}(\Rn)$, we have the inequality
\begin{equation}\label{Lp_Hardy_ineq1}
\left\|\frac{f}{|x|}\right\|_{p}\lesssim \|\nabla_{k}f\|_{p}.
\end{equation}
\end{thm}

\begin{remark}
In the case $p=2$, the Hardy inequality \eqref{Lp_Hardy_ineq1} was proven in full generality with respect to the range of the coefficient $\gamma$ in \cite{Vel19}. More precisely, if $N+2\gamma>2$, then the inequality
$$ \norm{\frac{f}{|x|}}_2 \lesssim \norm{\nabla_k f}_2$$
holds for all $f\in C_c^\infty(\RR^N)$. This observation allows us to extend the range of $\gamma$ if $p=2$ in all our results below where Theorem \ref{Lp_Hardy_thm} is employed. 
\end{remark}

\begin{remark} In \cite[Theorem 1.3]{GIT17} the authors prove a Stein-Weiss inequality for the Dunkl Riesz potential (\cite{TX07}), which implies the following Hardy-Sobolev inequality for the Dunkl Laplacian
$$ \norm{|x|^{-s} f}_p \lesssim \norm{(-\Delta_k)^{s/2}f}_p,$$
which holds for all $f$ in the Lizorkin space
$$\Phi_{k}=\left\{f\in \mathcal{S}(\Rn):\int_{\Rn}x^{n}f(x)\diff\mu_{k}(x)=0, \;n\in \mathbb{Z}_{+}^{N}\right\}.$$
\end{remark}

Now we are ready to prove Theorem \ref{uncer_thm}.
\begin{proof}[Proof of Theorem \ref{uncer_thm}] Using H\"{o}lder's inequality we obtain from \eqref{Lp_Hardy_ineq1} that
\begin{equation*}
    \begin{split}
    \left(\int_{\Rn}|\nabla_{k}f|^{p}\diff\mu_{k}\right)^{\frac{1}{p}}\left(\int_{\Rn}|x|^{q}|f|^{q}\diff\mu_{k}\right)^{\frac{1}{q}} &\gtrsim
    \left(\int_{\Rn}\frac{|f|^{p}}{|x|^{p}}\diff\mu_{k}\right)^{\frac{1}{p}}\left(\int_{\Rn}|x|^{q}|f|^{q}\diff\mu_{k}\right)^{\frac{1}{q}}\\&\gtrsim \int_{\Rn}|f|^{2}\diff\mu_{k},
    \end{split}
\end{equation*}
as desired.
\end{proof}

\section{Gagliardo-Nirenberg inequalities}
\label{SEC:GN}
In this section we prove different forms of the Gagliardo-Nirenberg inequality for Dunkl operators. Firstly, in Section \ref{subsection-GN1}, we prove the case $j=0$, $l=1$ of the classical Gagliardo-Nirenberg inequality \eqref{classicalGN}, which follows from the Sobolev inequality, and which is the main tool in the analysis of the nonlinear damped wave equation in the last section. Secondly, in Section \ref{subsection-GN2} we will prove the more involved case $p=q=r$, which relies on Littlewood-Paley theory. Finally, in Section \ref{subsection-GN3}, we prove several weighted forms of the Gagliardo-Nirenberg inequality.

\subsection{The Gagliardo-Nirenberg inequality I} \label{subsection-GN1}

We first prove the Dunkl analogue of the case $j=0$, $l=1$ of the classical Gagliardo-Nirenberg inequality (\ref{classicalGN}).

\begin{thm} \label{GNineq}
Let $1\leq p,q < \infty$, $1\leq r <N+2\gamma$, and $0\leq \theta \leq 1$ satisfying the assumption
\begin{equation} \label{assumptionGN}
\theta \left( \frac{1}{N+2\gamma} + \frac{1}{p} - \frac{1}{r} \right) = \frac{1}{p} - \frac{1}{q}.
\end{equation}
Then the following inequality holds for all $f\in \DSS \cap L^p(\mu_k)$
$$ \norm{f}_{q} \leq C \norm{\nabla_k f}_r^\theta \norm{f}_p^{1-\theta},$$
for some constant $C>0$.
\end{thm}

\begin{proof}
We note first that it is enough to prove the result for $f\in C_c^\infty(\RR^N)$ and then use a density argument. If $\theta =0$, then $p=q$ and the inequality reduces to the trivial case $\norm{f}_p \leq C \norm{f}_p$. For the rest of the proof we will assume that $\theta>0$.

Since $r<N+2\gamma$, from (\ref{assumptionGN}) we obtain $(1-\theta)q <p$. We can then define $s>0$ by the equation
\begin{equation} \label{firststepdefnGN}
\frac{q(1-\theta)}{p} + \frac{1}{s} =1.
\end{equation}
Then, by H\"older's inequality we have
\begin{equation} \label{GNproofHolder}
\begin{aligned}
\int_{\RR^N} |f|^q \diff\mu_k
&\leq \left(\int_{\RR^N} |f|^{qs\theta} \diff\mu_k \right)^{1/s} \left( \int_{\RR^N} |f|^p \diff\mu_k \right)^{\frac{q(1-\theta)}{p}}
\\
&= \norm{f}_{qs\theta}^{q\theta} \norm{f}_p^{q(1-\theta)}.
\end{aligned}
\end{equation}

At this point, we would like to bound $\norm{f}_{qs\theta}$ from above by $\norm{\nabla_k f}_{r}$, which, together with the above inequality, concludes the Theorem. This can be achieved by the Sobolev inequality \ref{sobolev} as long as $1\leq r < N+2\gamma$ and
$$ \frac{1}{qs\theta} = \frac{1}{r}-\frac{1}{N+2\gamma}.$$
This follows from the assumption (\ref{assumptionGN}) and the definition of $s$ in equation (\ref{firststepdefnGN}). Thus we can indeed apply the Sobolev inequality to obtain
\begin{equation} \label{GNproofSobolev}
\norm{f}_{qs\theta} \leq C \norm{\nabla_k f}_{r}.
\end{equation}
Therefore, from (\ref{GNproofHolder}) and (\ref{GNproofSobolev}) we can conclude that
$$ \norm{f}_q \leq C \norm{\nabla_k f}_r^\theta \norm{f}_p^{1-\theta},$$
which is what we needed to prove.
\end{proof}

We record in the next corollary the very useful particular case when $p=r=2$, while in the second corollary we prove a consequence of Gagliardo-Nirenberg involving the Sobolev norm that will be essential later in the study of the wave equation.

\begin{cor} \label{GNspecialcase}
Suppose $N+2\gamma>2$. Let $2\leq q \leq \frac{2(N+2\gamma)}{N+2\gamma-2}$ and
\begin{equation} \label{GNtheta}
\theta = (N+2\gamma) \frac{q-2}{2q}.
\end{equation}
Then there exists a constant $C>0$ such that the inequality
$$ \norm{f}_q \leq C \norm{\nabla_k f}_2^\theta \norm{f}_2^{1-\theta}$$
holds for any $f\in \DSS$.
\end{cor}

\begin{cor} \label{GNcorollary2}
Suppose $N+2\gamma >2$. Then for any $2\leq q \leq \frac{2(N+2\gamma)}{N+2\gamma-2}$ there exists a constant $C>0$ such that the inequality
$$ \norm{f}_q \leq C \norm{f}_{\DSS}$$
holds for any $f\in \DSS$.
\end{cor}

\begin{proof}
We use the Young inequality
$$ a^\theta b^{1-\theta} \leq \theta a + (1-\theta)b,$$
which holds for all $a,b\geq 0$ and $0\leq \theta \leq 1$. From the previous corollary, with $\theta$ as in (\ref{GNtheta}), we have
\begin{align*}
\norm{f}_q^2
\leq C \max\{\theta,1-\theta\} \cdot (\norm{\nabla_k f}_2^2 + \norm{f}_2^2)
\leq C \norm{f}_{\DSS}^2,
\end{align*}
and the result follows immediately.
\end{proof}

\subsection{The Gagliardo-Nirenberg inequality II} \label{subsection-GN2}

In this section we prove the case $p=q=r$ of the Gagliardo-Nirenberg inequality. Note that in this case the assumptions in (\ref{classicalGN}) imply that $\theta=\frac{j}{l}$. However, we adopt a slightly different notation in this section that will be more convenient for the present case. The main result is the following.

\begin{thm} [Gagliardo-Nirenberg inequality] \label{GN2}
Let $1<p<\infty$, $0 \leq s \leq \frac{N+2\gamma}{p}$ and $0\leq \theta\leq 1$. Then, for any $f \in \mathcal{S}(\RR^N)$, we have the inequality
\begin{equation} \label{GN1_ineq1}
 \norm{(-\Delta_k)^{s(1-\theta)/2}f}_p
	\leq C \norm{(-\Delta_k)^{s/2}f}_p^{1-\theta} \norm{f}_p^\theta,
\end{equation}
where $C>0$ is a constant.
\end{thm}

In order to prove this result, we need to introduce some Littlewood-Paley theory. Let $\psi \in C_c^\infty(\RR)$ with support contained in the annulus $\frac{1}{2} \leq |x| \leq 2$ and such that
\begin{equation} \label{lpsum1}
\sum_{j\in\ZZ} \psi(2^{-j}x) = 1 \qquad \text{ for all } x\in\RR\setminus \{0\}.
\end{equation}
For convenience denote $\psi_j (x) = \psi(2^{-j} x)$ and note that $\psi_j$ has support contained in the annulus $2^{j-1} \leq |x| \leq 2^{j+1}$. Define the following operators on $L^2(\mu_k)$
$$ P_j (f) = \DT^{-1} (\psi_j \DT(f)).$$
Then (\ref{lpsum1}) implies that
\begin{equation} \label{sumpj}
\sum_{j\in\ZZ} P_j = I,
\end{equation}
where $I$ is the identity operator on $L^2(\mu_k)$. The following result will be essential in the proof of Theorem \ref{GN2}.

%%%%%%%%%%%%%%%%%%%%%%%%%%%%%%%%%%%%%%%%%%%%%%%%%%%%%%%%%%%%%%%%%%%%%%%%%%%%%%%%%%%%%%%%%%%%%%%%%
\begin{prop} \label{littewoodpaleyfractional}
For any $s \geq 0$ and $1<p<\infty$ there exist two constants $C_1,C_2>0$ such that
$$ C_1 \norm{(-\Delta_k)^{s/2} f}_p \leq \norm{\left(\sum_{j\in\ZZ} |2^{js} P_j f|^2 \right)^{1/2}}_p
\leq C_2 \norm{(-\Delta_k)^{s/2} f}_p$$
holds for all $f\in \mathcal{S}(\RR^N)$.
\end{prop}
%%%%%%%%%%%%%%%%%%%%%%%%%%%%%%%%%%%%%%%%%%%%%%%%%%%%%%%%%%%%%%%%%%%%%%%%%%%%%%%%%%%%%%%%%%%%%%%%%

Before we prove the Proposition, we need the following two results. The first is the H\"ormander-Mikhlin multiplier theorem which was proved in the case of Dunkl operators in \cite{DW} (see also \cite[Chapter 7]{DX}). The second result is a Lemma which corresponds to the particular case $s=0$ in Proposition \ref{littewoodpaleyfractional}, but where the conditions on the function $\psi$ are relaxed.

%%%%%%%%%%%%%%%%%%%%%%%%%%%%%%%%%%%%%%%%%%%%%%%%%%%%%%%%%%%%%%%%%%%%%%%%%%%%%%%%%%%%%%%%%%%%%%%%%
\begin{thm}[H\"ormander-Mikhlin multiplier theorem - Theorem 4.1, \cite{DW}]

Let $m:(0,\infty) \to \RR$ be a function satisfying $\norm{m}_\infty \leq A$ and the H\"ormander's condition
$$ \frac{1}{R} \int_R^{2R} |m^{(r)}(x)| \diff x \leq A R^{-r} \qquad \forall R>0,$$
where $r$ is the smallest integer greater than or equal to $\frac{N+1}{2} + \gamma$. Let $\mathcal{T}_m$ be the operator on $L^2(\mu_k)$ defined by
$$ \DT(\mathcal{T}_mf)(\xi) = m(|\xi|) \DT(f)(\xi) \qquad \forall \xi\in\RR^N.$$
Then for all $1<p<\infty$ there exists a constant $C_p>0$ such that
$$ \norm{\mathcal{T}_mf}_p \leq C_pA \norm{f}_p$$
for all $f\in \mathcal{S}(\RR^N)$.
\end{thm}
%%%%%%%%%%%%%%%%%%%%%%%%%%%%%%%%%%%%%%%%%%%%%%%%%%%%%%%%%%%%%%%%%%%%%%%%%%%%%%%%%%%%%%%%%%%%%%%%%

%%%%%%%%%%%%%%%%%%%%%%%%%%%%%%%%%%%%%%%%%%%%%%%%%%%%%%%%%%%%%%%%%%%%%%%%%%%%%%%%%%%%%%%%%%%%%%%%%
\begin{lem} \label{littlewoodpaleylemma}
Let $\varphi\in C_c^\infty(\RR)$ with support contained in the annulus $\frac{1}{2} \leq |x| \leq 2$. As in the above, let $\varphi_j(x)=\varphi(2^{-j}x)$ and consider the operator on $L^2(\mu_k)$ given by
$$ P_j ^\varphi f = \DT^{-1} (\varphi_j \DT(f)).$$
Then, for any $1<p<\infty$ there exists a constant $C>0$ such that
\begin{equation} \label{littlewoodpaleyineq1}
\norm{\left( \sum_{j\in\ZZ} |P_j^\varphi f|^2 \right)^{1/2}}_p \leq C \norm{f}_p.
\end{equation}
\end{lem}
%%%%%%%%%%%%%%%%%%%%%%%%%%%%%%%%%%%%%%%%%%%%%%%%%%%%%%%%%%%%%%%%%%%%%%%%%%%%%%%%%%%%%%%%%%%%%%%%%

\begin{proof}
Let $\{\epsilon_j\}_{j\in\ZZ}$ be a sequence of independent identically distributed random variables with $\mathbb{P}(\epsilon_j = \pm 1) = \frac{1}{2}$. Recall Khinchin's inequality: if $1<p<\infty$, then there exist two constants $C_1,C_2>0$ such that for any $y_1,\ldots, y_n \in \CC$ we have
$$ C_1 \left(\sum_{j=1}^n |y_j|^2 \right)^{1/2}
\leq \left( \mathbb{E}\left| \sum_{j=1}^n \epsilon_j y_j \right|^p \right)^{1/p}
\leq C_2 \left(\sum_{j=1}^n |y_j|^2 \right)^{1/2}.$$

Consider then the operator
$$ \mathcal{T}f = \sum_{j\in\ZZ} \epsilon_j P_j^\varphi f.$$
Then
$$ \DT(\mathcal{T}f) = \sum_{j\in\ZZ} \epsilon_j \varphi_j \DT(f),$$
so the operator $\mathcal{T}$ is given by the multiplier
$$ m(x) = \sum_{j\in\ZZ} \epsilon_j \varphi_j(x).$$
We wish to apply the H\"ormander-Mikhlin multiplier theorem. Let $r$ be the smallest integer greater than or equal to $\frac{N+1}{2}+\gamma$. We have
$$ m^{(r)}(x) = \sum_{j\in\ZZ} 2^{-jr} \epsilon_j \varphi^{(r)}(2^{-j}x).$$
By considering the support of $\varphi$, at most two terms in this sum are non-zero, corresponding to the $j\in\ZZ$ such that
$$ \log_2 x - 1 \leq j \leq \log_2 x + 1.$$
Thus
$$ |m^{(r)}(x)| \leq C |2^{-jr}| \leq C x^{-r},$$
so
$$ \int_R^{2R} |m^{(r)}(x)| \diff x \leq C R^{-r+1},$$
so $m$ does indeed satisfy the requirements of the H\"ormander-Mikhlin theorem. Therefore
\begin{equation} \label{multiplierineq}
\norm{\mathcal{T}f}_p \leq C \norm{f}_p.
\end{equation}

On the other hand, by Khinchin's inequality we have
$$ \mathbb{E}(|\mathcal{T}f(x)|^p) = \mathbb{E} \left( \left| \sum_{j\in\ZZ} \epsilon_j P_j^\varphi f(x) \right|^p \right)
\geq C \left(\sum_{j\in\ZZ} |P_j^\varphi f(x)|^2 \right)^{p/2}. $$
But (\ref{multiplierineq}) implies that
$$ \mathbb{E}(\norm{\mathcal{T}f}_p^p) \leq C \norm{f}_p^p.$$
The last two inequalities finally imply (\ref{littlewoodpaleyineq1}).
\end{proof}

\begin{proof} [Proof of Proposition \ref{littewoodpaleyfractional}]
We will first prove the inequality on the right hand side. Since $\DT((-\Delta_k)^{s/2}f)(\xi) = |\xi|^{s} \DT(f)(\xi)$, we have
\begin{equation} \label{littlewoodpaleydt}
2^{js} P_jf(x) = \DT^{-1} (2^{js} |\xi|^{-s} \psi_j \DT((-\Delta_k)^{s/2} f)).
\end{equation}
Let $\varphi(\xi)=\frac{\psi(\xi)}{|\xi|}$ and let $P_j^\varphi$ be defined as in the previous lemma. Also let $g=(-\Delta_k)^{s/2}f$. Then from (\ref{littlewoodpaleydt}) it follows that
$$ \sum_{j\in\ZZ} |2^{js} P_j f|^2 = \sum_{j\in\ZZ} |P_j^\varphi g|^2,$$
so by Lemma \ref{littlewoodpaleylemma} we have
\begin{equation} \label{littlewoodpaleyright}
\norm{\left( \sum_{j\in\ZZ} |2^{js} P_j f|^2 \right)^{1/2}}_p
= \norm{\left( \sum_{j\in\ZZ} |P_j^\varphi g|^2 \right)^{1/2}}_p
\leq C \norm{g}_p
= C \norm{(-\Delta_k)^{s/2} f}_p.
\end{equation}

We now turn to the left hand side inequality. Let $g\in \mathcal{S}(\RR^N)$ be arbitrary. By considering the support of $\psi$ we can see that
$$ \IntN P_j f \cdot P_i g \diff\mu_k =0$$
for all $|j-i| >2$. Thus, by (\ref{sumpj}), we have
\begin{align*}
\IntN (-\Delta_k)^{s/2} f \overline{g} \diff\mu_k
& = \sum_{i=-2}^2 \sum_{j\in\ZZ} \IntN P_j((-\Delta_k)^{s/2} f) P_{j+i} \overline{g} \diff\mu_k
\\
&= \sum_{i=-2}^2 \sum_{j\in\ZZ} \IntN \psi_j(\xi) \psi_{j+i}(\xi) |\xi|^s \DT (f)(\xi) \DT (\overline{g})(\xi) \diff\mu_k(\xi)
\\
&= \sum_{i=-2}^2 \sum_{j\in\ZZ} \IntN 2^{js} P_j(f) P_{j+i}^\varphi (\overline{g}) \diff\mu_k,
\end{align*}
where $\varphi(\xi)=|\xi|^s \psi(\xi)$. Here we used the Plancherel identity twice. Thus, by the Cauchy-Schwarz inequality we have
\begin{align*}
\left| \IntN (-\Delta_k)^{s/2} f \overline{g} \diff\mu_k  \right|
&\leq \sum_{i=-2}^2 \IntN \left(\sum_{j\in\ZZ} |2^{js}P_jf|^2 \right)^{1/2} \left(\sum_{j\in\ZZ} |P_{j+i}^\varphi \overline{g}|^2 \right)^{1/2} \diff\mu_k
\\
&= 5 \IntN \left(\sum_{j\in\ZZ} |2^{js}P_jf|^2 \right)^{1/2} \left(\sum_{j\in\ZZ} |P_{j}^\varphi \overline{g}|^2 \right)^{1/2} \diff\mu_k.
\end{align*}
By H\"older's inequality and then Lemma \ref{littlewoodpaleylemma}, this implies that
\begin{align*}
\left| \IntN (-\Delta_k)^{s/2} f \overline{g} \diff\mu_k  \right|
&\leq 5 \norm{\left(\sum_{j\in\ZZ} |2^{js}P_jf|^2 \right)^{1/2}}_p \cdot \norm{\left(\sum_{j\in\ZZ} |P_{j}^\varphi \overline{g}|^2 \right)^{1/2}}_{p'}
\\
&\leq C \norm{g}_{p'} \cdot \norm{\left(\sum_{j\in\ZZ} |2^{js}P_jf|^2 \right)^{1/2}}_p.
\end{align*}
Finally, since $g\in \mathcal{S}(\RR^N)$ was arbitrary, it follows that
$$ \norm{(-\Delta_k)^{s/2} f}_p \leq C \norm{\left(\sum_{j\in\ZZ} |2^{js}P_jf|^2 \right)^{1/2}}_p.$$
This completes the proof of Proposition \ref{littewoodpaleyfractional}.
\end{proof}

We are now ready to prove the main result of this section.

\begin{proof} [Proof of Theorem \ref{GN2}] It is easy to see that the cases $\theta=0,1$ are trivial. So, for $0<\theta<1$ by Proposition \ref{littewoodpaleyfractional}, we have
\begin{align*}
\norm{(-\Delta_k)^{s(1-\theta)/s} f}_p
\leq C \norm{ \left(\sum_{j\in\ZZ} |2^{js(1-\theta)} P_jf|^2 \right)^{1/2}}_p.
\end{align*}
But by H\"older's inequality we have
$$ \sum_{j\in\ZZ} |2^{js(1-\theta)} P_jf|^2
\leq \left( \sum_{j\in\ZZ} |2^{js} P_j f|^2 \right)^{1-\theta} \cdot \left( \sum_{j\in\ZZ} |P_j f|^2 \right)^\theta. $$
Applying H\"older's inequality again, the last two results imply
\begin{align*}
\norm{(-\Delta_k)^{s(1-\theta)/s} f}_p
&\leq C \norm{\left( \sum_{j\in\ZZ} |2^{js} P_j f|^2 \right)^{1/2}}_p^{1-\theta} \cdot \norm{\left( \sum_{j\in\ZZ} |P_j f|^2 \right)^{1/2}}_p^\theta
\\
&\leq C \norm{(-\Delta_k)^{s/2} f}_p^{1-\theta} \norm{f}_p^\theta,
\end{align*}
where in the last step we used Proposition \ref{littewoodpaleyfractional}.
\end{proof}

\subsection{Weighted Gagliardo-Nirenberg Inequalities} \label{subsection-GN3}

We will now prove a number of weighted Gagliardo-Nirenberg inequalities.

\begin{thm}\label{GN2_thm} Let $1<p<\frac{N+2\gamma}{2+2\gamma}$, $2\leq s\leq \frac{N+2\gamma}{p}$, and define $q=\frac{p(N+2\gamma)}{N+2\gamma-p}$. Then, for any $f\in C_c^\infty(\RR^N)$
we have
\begin{equation}\label{GN2_ineq1}
\left\|\frac{f}{|x|}\right\|_{q}\lesssim \|(-\Delta_{k})^{s/2}f\|_{p}^{2/s}\|f\|^{1-2/s}_{p}.
\end{equation}
\end{thm}
\begin{proof}[Proof of Theorem \ref{GN2_thm}] From the condition $1<p<\frac{N+2\gamma}{2+2\gamma}$, we obtain $1<q<\frac{N+2\gamma}{1+2\gamma}$, so we can apply Theorem \ref{Lp_Hardy_thm} to obtain
\begin{equation}\label{GN2_ineq2}
\left\|\frac{f}{|x|}\right\|_{q}\lesssim \|\nabla_{k}f\|_{q}.
\end{equation}
Since $1<p<\frac{N+2\gamma}{1+2\gamma} \leq N+2\gamma$ and $q=\frac{p(N+2\gamma)}{N+2\gamma-p}$, we can apply the Sobolev inequality of Theorem \ref{sobolev} to the right hand side of \eqref{GN2_ineq2} to obtain
\begin{equation}\label{GN2_ineq3}
\left\|\frac{f}{|x|}\right\|_{q}\lesssim \|\Delta_{k}f\|_{p}.
\end{equation}
Now, since $2\leq s\leq \frac{N+2\gamma}{p}$, applying \eqref{GN1_ineq1} with $\theta=1-2/s$ to the right hand side of \eqref{GN2_ineq3} we obtain \eqref{GN2_ineq1}, as required.
\end{proof}

\begin{thm}\label{GN3_thm} Assume that $\frac{N+2\gamma}{2}>2$. Let $1\leq a\leq 2$, $2\leq s\leq \frac{N+2\gamma}{2}$, and define $p=\frac{2(N+2\gamma)}{N+2\gamma-2+2(a-1)}$. Then, for any $f\in C_c^\infty(\RR^N)$
we have
\begin{equation}\label{GN3_ineq1}
\left\|\frac{f}{|x|^{a}}\right\|_{p}\lesssim \|(-\Delta_{k})^{s/2}f\|_{2}^{2/s}\|f\|^{1-2/s}_{2}.
\end{equation}
\end{thm}

\begin{proof}[Proof of Theorem \ref{GN3_thm}]
Since $1\leq a\leq 2$, $\frac{N+2\gamma-2}{2}>1$ and $p=\frac{2(N+2\gamma)}{N+2\gamma-2+2(a-1)}$, then by the weighted Rellich inequality \eqref{weigh_Rellich_ineq1} we have
\begin{equation}\label{GN3_ineq2}
\left\|\frac{f}{|x|^{a}}\right\|_{p}\lesssim \|\Delta_{k}f\|_{2}.
\end{equation}
On the other hand, we have by \eqref{GN1_ineq1} that
\begin{equation}\label{GN3_ineq3}
\|\Delta_{k}f\|_{2}\lesssim \|(-\Delta_{k})^{s/2}f\|_{2}^{2/s}\|f\|^{1-2/s}_{2}
\end{equation}
since $2\leq s\leq \frac{N+2\gamma}{2}$. Then, these inequalities \eqref{GN3_ineq2} and \eqref{GN3_ineq3} give \eqref{GN3_ineq1}, as required.
\end{proof}

For the next result, we will need the following Hardy inequality for fractional Dunkl Laplacian from \cite{Vel19}.

\begin{thm}\label{frac_Hardy_thm} Let $0\leq s< \frac{N+2\gamma}{2}$. Then for all $f\in \mathcal{S}(\Rn)$ we have
\begin{equation}\label{frac_Hardy_eq1}
C(s)\left\|\frac{f}{|x|^{s}}\right\|_{2}\leq \|(-\Delta_{k})^{s/2}f\|_{2},
\end{equation}
with the sharp constant $C(s)=2^{s}\frac{\Gamma\left(\frac{1}{2}\left(\frac{N}{2}+\gamma+s\right)\right)}{\Gamma\left(\frac{1}{2}\left(\frac{N}{2}+\gamma-s\right)\right)}$.
\end{thm}

\begin{thm}\label{GN4_thm} Let $0\leq a \leq s \leq \frac{N+2\gamma}{2}$. Then, for any $f\in\mathcal{S}(\Rn)$
we have
\begin{equation}\label{GN4_ineq1}
\left\|\frac{f}{|x|^{a}}\right\|_{2}\lesssim \|(-\Delta_{k})^{s/2}f\|_{2}^{a/s}\|f\|^{1-a/s}_{2}.
\end{equation}
\end{thm}

\begin{remark} In the case when $a=s$ this result implies Theorem \ref{frac_Hardy_thm}, while $a=s=2$ this theorem gives Rellich inequality \cite[Theorem 7.1]{Vel19}.
\end{remark}

\begin{proof}[Proof of Theorem \ref{GN4_thm}] Since we have $0\leq a\leq \frac{N+2\gamma}{2}$ one gets from Theorem \ref{frac_Hardy_thm} that
\begin{equation}\label{GN4_ineq2}
\left\|\frac{f}{|x|^{a}}\right\|_{2}\lesssim \|(-\Delta_{k})^{a/2}f\|_{2}.
\end{equation}
On the other hand, we have by \eqref{GN1_ineq1} that
\begin{equation}\label{GN4_ineq3}
\|(-\Delta_{k})^{a/2}f\|_{2}\lesssim \|(-\Delta_{k})^{s/2}f\|_{2}^{a/s}\|f\|^{1-a/s}_{2}
\end{equation}
since $a \leq s \leq \frac{N+2\gamma}{2}$. Then, inequalities \eqref{GN4_ineq2} and \eqref{GN4_ineq3} give \eqref{GN4_ineq1}, as required.
\end{proof}

\section{Trudinger's inequality} \label{SEC:Trudinger}

In this section we prove the following improved version of the Trudinger inequality for Dunkl operators. The approach follows that of \cite{Ozawa95}.

\begin{thm}\label{trudinger}
Let $1<p<\infty$, and $\frac{1}{p}+\frac{1}{p'}=1$. Then, there exist constants $a,C>0$ such that for any $f\in\mathcal{S}(\RR^N)$ that satisfies the assumption
\begin{equation} \label{trudinger_assumption}
\norm{(-\Delta_k)^{(N+2\gamma)/(2p)} f}_p \leq 1,
\end{equation}
we have the inequality
\begin{equation} \label{trudinger_ineq}
\int_{\RR^N} \left( \exp(a |f(x)|^{p'}) - \sum_{\substack{ 0\leq j <p-1 \\ j\in\NN}} \frac{1}{j!} (a |f(x)|^{p'})^j \right) \diff\mu_k(x) 
\leq C \norm{f}_p^p.
\end{equation}
\end{thm}

Before we prove this theorem, we need to an auxiliary result and some background on Riesz potentials.

The Riesz potential for Dunkl operators was defined in \cite{TX07} for $0<s <N+2\gamma$ and for all $f\in \mathcal{S}(\RR^N)$ as 
$$ \mathcal{I}_s (f)(x):= \frac{1}{A_s} \IntN \tau_y f(x) \frac{1}{|y|^{N+2\gamma-s}} \diff\mu_k(y),$$
where $\tau_y$ is the generalised translation operator, and $a_s = 2^{s - \frac{N+2\gamma}{2}} \frac{\Gamma(s/2)}{\Gamma\left(\frac{N+2\gamma-s}{2}\right)}$. In \cite{HMS} and \cite{GIT2017} it was shown that for $1<p<q<\infty$ and $0<s<N+2\gamma$ such that $\frac{1}{p}-\frac{1}{q} = \frac{s}{N+2\gamma}$, we have the inequality
$$ \norm{\mathcal{I}_s f}_q \leq C_{p,q} \norm{f}_p \qquad \forall f \in \mathcal{S}(\RR^N).$$
Following the proof of \cite{GIT2017} carefully, we can estimate the behaviour of the constant $C_{p,q}$ with respect to $q$. More precisely, we have the following result.
 
\begin{prop} \label{RieszPotentialBound}
Let $1<p<q<\infty$ and $0<s<N+2\gamma$ defined by $\frac{1}{p}-\frac{1}{q} = \frac{s}{N+2\gamma}$. Assume that there exists a constant $c(p)>1$ such that $c(p)p <q$. Then, we have
\begin{equation} \label{RieszPotentialBound_eqn} 
\norm{\mathcal{I}_s f}_q \leq C(p) q^{1-\frac{1}{p}} \norm{f}_p \qquad \forall f \in \mathcal{S}(\RR^N),
\end{equation}
for a constant $C(p) >0$ which depends only on $p$.
\end{prop}

\begin{proof}
From our assumption we have
$$ 0 < \frac{N+2\gamma}{p} \left( 1- \frac{1}{c(p)}\right)< s < \frac{N+2\gamma}{p},$$
so 
$$ A_s > d(p) >0,$$
where $d(p)$ is a constant that depends only on $p$.

Following the proof of \cite{GIT2017}, we estimate
\begin{equation} \label{Rieszpotentialestimate1} 
|\mathcal{I}_s(f)(x)| \leq C_1(p) \left[ \frac{pq}{q-p} R^s M_kf(x) + q^{1-\frac{1}{p}} R^{-\frac{N+2\gamma}{q}} \norm{f}_p \right],
\end{equation}
which holds for all $R>0$, and where $C_1(p)>0$ is a constant depending only on $p$. Here $M_kf$ is the maximal function associated to $f$ and it is known that $M_k$ is bounded operator on $L^p(\mu_k)$ for $1<p\leq \infty$. See \cite{TX05} for the definition of the maximal function and a proof of this fact.

Using our assumption again, we have
$$ \frac{q}{q-p} > \frac{1}{1-\frac{1}{c(p)}},$$
and so \eqref{Rieszpotentialestimate1} becomes
\begin{equation} \label{Rieszpotentialestimate2} 
|\mathcal{I}_s(f)(x)| \leq C_2(p) q^{1-\frac{1}{p}}\left[ M_kf(x) +  R^{-\frac{N+2\gamma}{q}} \norm{f}_p \right].
\end{equation}
This holds for all $R>0$ and it is optimised for $R=\left( \frac{\norm{f}_p}{M_kf(x)}\right)^{\frac{p}{N+2\gamma}}$, thus obtaining
$$  \norm{\mathcal{I}_s f}_q \leq C_3(p) q^{1-\frac{1}{p}} \norm{M_kf}_p^{\frac{p}{q}} \norm{f}_p^{1-\frac{p}{q}} \leq C_4(p) q^{1-\frac{1}{p}} \norm{f}_p,$$
for constants $C_3(p), C_4(p) >0$ depending only on $p$. This completes the proof.
\end{proof}

In order to be able to use this proposition, we need to invert the Riesz potential and express the inequality in terms of the fractional Dunkl laplacian. This is done in the following proposition, which is a simple consequence of the previous result, but it requires some facts about the Riesz potential.

\begin{prop} \label{RieszPotentialBound2}
Let $1<p<q<\infty$ such that there exists a constant $c(p)>1$ and we have $c(p)p<q$. Then, we have the inequality
$$ \norm{f}_q \leq C(p)q^{1-\frac{1}{p}} \norm{(-\Delta_k)^{\frac{N+2\gamma}{2}\left( \frac{1}{p} -\frac{1}{q}\right)} f}_p \qquad \forall f\in\mathcal{S}(\RR^N),$$
where $C(p)>0$ is a constant depending only on $p$.
\end{prop}

\begin{proof}
Taking the function $g=(-\Delta_k)^{\frac{N+2\gamma}{2}\left( \frac{1}{p} -\frac{1}{q}\right)} f \in \mathcal{S}(\RR^N)$ in \eqref{RieszPotentialBound_eqn}, one obtains
\begin{equation} \label{RieszPotentialBound2_step}
\norm{\mathcal{I}_s \left((-\Delta_k)^{\frac{s}{2}} f\right)}_q \leq C(p)q^{1-\frac{1}{p}} \norm{(-\Delta_k)^{\frac{N+2\gamma}{2}\left( \frac{1}{p} -\frac{1}{q}\right)} f}_p,
\end{equation}
where $s=(N+2\gamma) \left( \frac{1}{p} - \frac{1}{q} \right)$.

It is known (see \cite{TX07}) that 
$$ \DT(\mathcal{I}_sF) (\xi) = |\xi|^{-s} \DT(\mathcal{I}_s F)(\xi)$$
in the sense of tempered distributions, i.e.,
$$ \IntN \mathcal{I}_s F(x) \overline{G(x)} \diff\mu_k (x)
= \IntN |\xi|^{-s} \DT(F)(\xi) \overline{\DT(G)(\xi)} \diff\mu_k(\xi),$$
for all $F,G \in \mathcal{S}(\RR^N)$. Taking $F=(-\Delta_k)^{\frac{s}{2}} f$ in this equality, we have
\begin{align*}
\IntN \mathcal{I}_s \left((-\Delta_k)^{\frac{s}{2}} f \right) (x) \overline{G(x)} \diff\mu_k(x)
&=\IntN |\xi|^{-s} \DT\left((-\Delta_k)^{\frac{s}{2}} f \right)(\xi) \overline{\DT(G)}(\xi) \diff\mu_k(\xi)
\\
&=\IntN \DT(f)(\xi) \overline{\DT(G)}(\xi) \diff\mu_k(\xi)
\\
&=\IntN f(x) \overline{G(x)} \diff\mu_k(x),
\end{align*}
where in the last line we used Parseval's identity. This holds for all $G\in\mathcal{S}(\RR^N)$, thus by density we have
$$\IntN \mathcal{I}_s \left((-\Delta_k)^{\frac{s}{2}} f \right) (x) \overline{G(x)} \diff\mu_k(x) 
=\IntN f(x) \overline{G(x)} \diff\mu_k(x) \qquad \forall G\in L^{q'}(\mu_k),$$
and therefore
$$ \norm{\mathcal{I}_s \left((-\Delta_k)^{\frac{s}{2}} f \right)}_q = \norm{f}_q.$$
This, together with \eqref{RieszPotentialBound2_step}, concludes the proof.
\end{proof}

We are now ready to prove Trudinger's inequality for Dunkl operators.

\begin{proof}[Proof of Theorem \ref{trudinger}]
Expanding the exponential on the left hand side of \eqref{trudinger_ineq}, we have
\begin{align} \label{trudinger_proof_expansion}
\int_{\RR^N} \left( \exp(a |f(x)|^{p'}) - \sum_{\substack{ 0\leq j <p-1 \\ j\in\NN}} \frac{1}{j!} (a |f(x)|^{p'})^j \right) \diff\mu_k(x)
= \sum_{\substack{j \geq p-1 \\ j\in\NN}} \frac{1}{j!} a^j \norm{f}_{p'j}^{p'j}.
\end{align}

If $p\not\in \NN$, then $q:=p'j \geq \frac{\lceil p \rceil -1}{p-1}p$, where $\lceil p \rceil$ is the least integer greater of equal to $p$. We can then apply Proposition \ref{RieszPotentialBound2} with $c(p)=\frac{\lceil p \rceil -1}{p-1} >1$ to deduce
$$ \norm{f}_{p'j} \leq C(p) (p'j)^{1-\frac{1}{p}} \norm{(-\Delta_k)^{\frac{N+2\gamma}{2p}\left( 1-\frac{p}{p'j} \right)}f}_p.$$
Then, from the Gagliardo-Nirenberg inequality \eqref{GN1_ineq1} with $s=\frac{N+2\gamma}{p}$ and $\theta=\frac{p}{p'j}$, and using assumption \eqref{trudinger_assumption}, this becomes
\begin{equation} \label{trudinger_p'jnorm} 
\norm{f}_{p'j} \leq C(p) (p'j)^{1-\frac{1}{p}} \norm{f}_p^{\frac{p}{p'j}}.
\end{equation}

If $p\in\NN$ and $j=p-1$, then $p'j=p$, so the first term in the sum on the right hand side is simply $\frac{1}{(p-1)!}a^{p-1} \norm{f}_p^p$. For $j\geq p$, we have $q:=p'j \geq \frac{p^2}{p-1}$, and as above we obtain \eqref{trudinger_p'jnorm}.

Replacing inequality \eqref{trudinger_p'jnorm} back in \eqref{trudinger_proof_expansion}, we obtain
$$\int_{\RR^N} \left( \exp(a |f(x)|^{p'}) - \sum_{\substack{ 0\leq j <p-1 \\ j\in\NN}} \frac{1}{j!} (a |f(x)|^{p'})^j \right) \diff\mu_k(x)
\leq C \norm{f}_p^p,$$
where
$$ C= \sum_{\substack{j \geq p-1 \\ j\in\NN}} \frac{j^j}{j!} (a C(p)^{p'} p')^j.$$
Note that if $aC(p)^{p'}p' < \frac{1}{e}$, i.e., if $a<\frac{1}{ep'C(p)^{p'}}$, then $C<\infty$.
\end{proof}

\section{Caffarelli-Kohn-Nirenberg inequalities}
\label{SEC:CKN} In this section, we introduce Caffarelli-Kohn-Nirenberg type inequalities for Dunkl operators, and explain how they imply \eqref{CKN1_ineq1_intro}-\eqref{CKN3_ineq1_intro}.

Let us start with the following Dunkl analogue of \eqref{CKN1_ineq1_intro}:
\begin{thm}\label{CKN1_thm} Let $1<p<\frac{N+2\gamma}{1+2\gamma}$, $1<q<\infty$, $0<r<\infty$ with $p+q\geq r$. Let also $\delta\in[0,1]\cap\left[\frac{r-q}{r},\frac{p}{r}\right]$ and $b$, $c\in\mathbb{R}$. Assume that
$$\frac{\delta r}{p}+\frac{(1-\delta)r}{q}=1 \quad \text{ and } \quad c=-\delta+b(1-\delta).$$
Then we have the inequality
\begin{equation}\label{CKN1_ineq1}
\||x|^{c}f\|_{r}\lesssim \|\nabla_{k}f\|^{\delta}_{p}\||x|^{b}f\|^{1-\delta}_{q}
\end{equation}
for all $f\in C_c^{\infty}(\Rn)$.
\end{thm}
%%%%%%%%%%%%%%%%%%%%%%%%%%%%%%%%%%%%%%

\begin{proof}[Proof of Theorem \ref{CKN1_thm}] The cases $\delta=0$ and $\delta=1$ are trivial. When $\delta=0$ we have $r=q$ and $b=c$ from $\frac{\delta r}{p}+\frac{(1-\delta)r}{q}=1$ and $c=-\delta+b(1-\delta)$, respectively, so the inequality becomes $\norm{|x|^c f}_q \leq C \norm{|x|^c f}_q$. When $\delta=1$ we get $r=p$ and $c=-1$ from $\frac{\delta r}{p}+\frac{(1-\delta)r}{q}=1$ and $c=-\delta+b(1-\delta)$, respectively; this is the exact case of Theorem \ref{Lp_Hardy_thm}.

Therefore, let us prove the theorem for the case $\delta\in(0,1)\cap\left[\frac{r-q}{r},\frac{p}{r}\right]$. Noting $c=-\delta+b(1-\delta)$, a direct calculation gives
$$\||x|^{c}f\|_{r}=
\left(\int_{\Rn}|x|^{cr}|f(x)|^{r}\diff\mu_k\right)^{\frac{1}{r}}
=\left(\int_{\Rn}\frac{|f(x)|^{\delta r}}{|x|^{\delta r }}\cdot \frac{|f(x)|^{(1-\delta)r}}{|x|^{-br(1-\delta)}}\diff\mu_{k}\right)^{\frac{1}{r}}.$$
Since $\delta\in(0,1)\cap\left[\frac{r-q}{r},\frac{p}{r}\right]$ and $p+q\geq r$, then using H\"{o}lder's inequality for $\frac{\delta r}{p}+\frac{(1-\delta)r}{q}=1$, one has
\begin{equation} \label{CKN1_ineq2}
\begin{aligned}
\||x|^{c}f\|_{r}
&\leq \left(\int_{\Rn}\frac{|f(x)|^{p}}{|x|^{p}}\diff\mu_{k}\right)^{\frac{\delta}{p}}
\left(\int_{\Rn}\frac{|f(x)|^{q}}{|x|^{-bq}}\diff\mu_{k}\right)^{\frac{1-\delta}{q}}
\\
&=\left\|\frac{f}{|x|}\right\|^{\delta}_{p}
\left\|\frac{f}{|x|^{-b}}\right\|^{1-\delta}_{q}.
\end{aligned}
\end{equation}
Since $1<p<\frac{N+2\gamma}{1+2\gamma}$, applying Theorem \ref{Lp_Hardy_thm} to the right hand side of \eqref{CKN1_ineq2} we obtain \eqref{CKN1_ineq1}, as required.
\end{proof}

%%%%%%%%%%%%%%%%%%%%%%%%%%%%%%%%%%%%%%%%%%%%%%%%%%%%%%%%%%%%%%%%%%%%%%%%%%%%%%%%%%%%%%%%%%%%%%%%%
Now we state the following Dunkl analogue of the inequality \eqref{CKN2_ineq1_intro}:
\begin{thm}\label{CKN2_thm} Let $q,r,a,b,c,\delta \in \RR$ such that $1<q<\infty$, $0<r<\infty$,
$$\frac{2(N+2\gamma)}{N+2\gamma-2}+q\geq r, \quad \delta\in[0,1]\cap\left[\frac{r-q}{r},\frac{2(N+2\gamma)}{r(N+2\gamma-2)}\right] \quad \text{and} \;\; N+2\gamma-2+2a>0.$$
Assume also that
$$\frac{\delta r(N+2\gamma-2)}{2(N+2\gamma)}+\frac{(1-\delta)r}{q}=1 \quad \text{ and } \quad c=\delta a+b(1-\delta).$$
Then we have the following Caffarelli-Kohn-Nirenberg type inequality which holds for all $f\in C_c^{\infty}(\Rn)$:
\begin{equation}\label{CKN2_ineq1}
\||x|^{c}f\|_{r}\lesssim \||x|^{a}\nabla_{k}f\|^{\delta}_{2}\||x|^{b}f\|^{1-\delta}_{q}.
\end{equation}
\end{thm}
%%%%%%%%%%%%%%%%%%%%%%%%%%%%%%%%%%%%%%%%%%%%%%%%%%%%%%%%%%%%%%%%%%%%%%%%%%%%%%%%%%%%%%%%%%%%%%%%%

\begin{proof}[Proof of Theorem \ref{CKN2_thm}] The case $\delta=0$ is trivial since we have $r=q$ and $b=c$ from $\frac{\delta r(N+2\gamma-2)}{2(N+2\gamma)}+\frac{(1-\delta)r}{q}=1$ and $c=\delta a+b(1-\delta)$, respectively, so the inequality becomes $\norm{|x|^b f}_{q} \lesssim \norm{|x|^b f}_{q}$.

In the case $\delta=1$ we get $r=\frac{2(N+2\gamma)}{N+2\gamma-2}$ and $c=a$ from $\frac{\delta r(N+2\gamma-2)}{2(N+2\gamma)}+\frac{(1-\delta)r}{q}=1$ and $c=\delta a+b(1-\delta)$, respectively. Since we also have $N+2\gamma-2+2a>0$, this case follows from Theorem \ref{weigh_Hardy_thm}.

Now we consider the case $\delta\in(0,1)\cap\left[\frac{r-q}{r},\frac{2(N+2\gamma)}{r(N+2\gamma-2)}\right]$. Taking into account $c=\delta a+b(1-\delta)$, we have
$$\||x|^{c}f\|_{r}=
\left(\int_{\Rn}|x|^{cr}|f(x)|^{r} \diff \mu_k\right)^{\frac{1}{r}}
=\left(\int_{\Rn}\frac{|f(x)|^{\delta r}}{|x|^{-\delta ar }}\cdot \frac{|f(x)|^{(1-\delta)r}}{|x|^{-br(1-\delta)}}\diff\mu_{k}\right)^{\frac{1}{r}}.$$
Since $\delta\in(0,1)\cap\left[\frac{r-q}{r},\frac{2(N+2\gamma)}{r(N+2\gamma-2)}\right]$ and $\frac{2(N+2\gamma)}{N+2\gamma-2}+q\geq r$, then by H\"{o}lder's inequality for $\frac{\delta r(N+2\gamma-2)}{2(N+2\gamma)}+\frac{(1-\delta)r}{q}=1$, one has
\begin{equation} \label{CKN2_ineq2}
\begin{aligned}
\||x|^{c}f\|_{r}
&\leq \left(\int_{\Rn}\frac{|f(x)|^{\frac{2(N+2\gamma)}{N+2\gamma-2}}}{|x|^{\frac{-2(N+2\gamma)a}{N+2\gamma-2}}}\diff\mu_{k}\right)^{\frac{\delta (N+2\gamma-2)}{2(N+2\gamma)}}
\left(\int_{\Rn}\frac{|f(x)|^{q}}{|x|^{-bq}}\diff\mu_{k}\right)^{\frac{1-\delta}{q}}
\\
&=\left\|\frac{f}{|x|^{-a}}\right\|^{\delta}_{\frac{2(N+2\gamma)}{N+2\gamma-2}}
\left\|\frac{f}{|x|^{-b}}\right\|^{1-\delta}_{q}.
\end{aligned}
\end{equation}
On the other hand, since $N+2\gamma-2+2a>0$, by Theorem \ref{weigh_Hardy_thm} we have
\begin{equation}\label{CKN2_ineq3}
 \left\|\frac{f}{|x|^{-a}}\right\|^{\delta}_{\frac{2(N+2\gamma)}{N+2\gamma-2}}\lesssim \left\|\frac{\nabla_{k}f}{|x|^{-a}}\right\|^{\delta}_{2}
 \end{equation}
for $f\in C_c^{\infty}(\Rn)$. Thus, combining \eqref{CKN2_ineq2} and \eqref{CKN2_ineq3} we obtain \eqref{CKN2_ineq1}, as required.
\end{proof}

Finally, we conclude this section with a Caffarelli-Kohn-Nirenberg inequality for fractional Dunkl Laplacian, which is a generalisation of \eqref{CKN3_ineq1_intro}.

%%%%%%%%%%%%%%%%%%%%%%%%%%%%%%%%%%%%%%%%%%%%%%%%%%%%%%%%%%%%%%%%%%%%%%%%%%%%%%%%%%%%%%%%%%%%%%%%%
\begin{thm}\label{CKN3_thm} Let $a,b,c \in \RR$, $\delta\in[0,1]\cap\left[\frac{r-q}{r},\frac{2}{r}\right]$, $1<q<\infty$ and $0<r<\infty$ with $2+q\geq r$. Assume that
$$\frac{\delta r}{2}+\frac{(1-\delta)r}{q}=1 \quad \text{ and } \quad c=\delta (a-1)+b(1-\delta),$$
and also $1-\frac{N+2\gamma}{2}<a\leq 1$. Then we have the inequality
\begin{equation}\label{CKN3_ineq1}
\||x|^{c}f\|_{r}\leq\frac{1}{C(a)^{\delta}} \|(-\Delta_{k})^{(1-a)/2}f\|^{\delta}_{2}\||x|^{b}f\|^{1-\delta}_{q}
\end{equation}
for all $f\in \mathcal{S}(\Rn)$, where $C(a)=2^{1-a}\frac{\Gamma\left(\frac{1}{2}\left(\frac{N}{2}+\gamma+1-a\right)\right)}
{\Gamma\left(\frac{1}{2}\left(\frac{N}{2}+\gamma-1+a\right)\right)}$.
\end{thm}
\begin{remark}
In the special case when $\gamma=0$ we get that $\nabla_{k}=\nabla$, the usual gradient in $\Rn$. Then, in this case we see that Theorems \ref{CKN1_thm}-\ref{CKN3_thm} imply \eqref{CKN1_ineq1_intro}-\eqref{CKN3_ineq1_intro}, respectively.
\end{remark}
%%%%%%%%%%%%%%%%%%%%%%%%%%%%%%%%%%%%%%%%%%%%%%%%%%%%%%%%%%%%%%%%%%%%%%%%%%%%%%%%%%%%%%%%%%%%%%%%%

\begin{proof}[Proof of Theorem \ref{CKN3_thm}] In the case $\delta=0$ we have $r=q$ and $b=c$ from $\frac{\delta r}{2}+\frac{(1-\delta)r}{q}=1$ and $c=\delta (a-1)+b(1-\delta)$, so we are in the trivial case $\norm{|x|^b f}_{q} \leq  \norm{|x|^b f}_{q}$.

When $\delta=1$, we get $r=2$ and $c=a-1$ from $\frac{\delta r}{2}+\frac{(1-\delta)r}{q}=1$ and $c=\delta (a-1)+b(1-\delta)$, respectively. Then, it is easy to see that this the exact case of Theorem \ref{frac_Hardy_thm} since $1-\frac{N+2\gamma}{2}<a\leq 1$.

Now let us consider the case $\delta\in(0,1)\cap\left[\frac{r-q}{r},\frac{2}{r}\right]$. Noting that $c=\delta(a-1)+b(1-\delta)$, a direct calculation gives
$$\||x|^{c}f\|_{r}=
\left(\int_{\Rn}|x|^{cr}|f(x)|^{r}\diff x\right)^{\frac{1}{r}}
=\left(\int_{\Rn}\frac{|f(x)|^{\delta r}}{|x|^{\delta r(1-a) }}\cdot \frac{|f(x)|^{(1-\delta)r}}{|x|^{-br(1-\delta)}}\diff\mu_{k}\right)^{\frac{1}{r}}.$$
Since $\delta\in(0,1)\cap\left[\frac{r-q}{r},\frac{2}{r}\right]$ and $2+q\geq r$, then using H\"{o}lder's inequality for $\frac{\delta r}{2}+\frac{(1-\delta)r}{q}=1$, one has
\begin{equation} \label{CKN3_ineq2}
\begin{aligned}
\||x|^{c}f\|_{r}
&\leq \left(\int_{\Rn}\frac{|f(x)|^{2}}{|x|^{2(1-a)}}\diff\mu_{k}\right)^{\frac{\delta}{2}}
\left(\int_{\Rn}\frac{|f(x)|^{q}}{|x|^{-bq}}\diff\mu_{k}\right)^{\frac{1-\delta}{q}}
\\
&=\left\|\frac{f}{|x|^{1-a}}\right\|^{\delta}_{2}
\left\|\frac{f}{|x|^{-b}}\right\|^{1-\delta}_{q}
\end{aligned}
\end{equation}
Since $1-\frac{N+2\gamma}{2}<a\leq 1$, then applying Theorem \ref{frac_Hardy_thm} to the right hand side of \eqref{CKN3_ineq2} we obtain \eqref{CKN3_ineq1}, as required.
\end{proof}

\section{An application: nonlinear damped wave equations for the Dunkl Laplacian}
\label{SEC:Appl}

In this section we use the Gagliardo-Nirenberg inequality of Theorem \ref{GNineq} obtained above to study the nonlinear damped wave equation \eqref{Cauchy_problem}. As a first step, in Section \ref{linearwaveeq} we study the linear version of this equation, while in Section \ref{nonlinearwaveeq} we prove the main result that concerns global existence and uniqueness of solution for small data. This follows the treatment of \cite{RT}, where similar problems are considered for Rockland operators on Lie groups.

\subsection{Linear Damped Wave Equation} \label{linearwaveeq}

Let us first consider the Cauchy problem on $(t,x)\in \RR_+ \times \RR^N$
\begin{equation} \label{linearwave}
\begin{cases}
\partial_{tt}^2 u(t,x) - \Delta_k u(t,x) + b\partial_t u(t,x) +m u(t,x) =0 \\
u(0,x) = u_0(x)\\
\partial_t u(0,x) = u_1(x),
\end{cases}
\end{equation}
where $b,m>0$. We take Dunkl transform with respect to the variable $x$ in this equation to obtain
\begin{equation} \label{linearwaveDT}
\begin{cases}
\partial_{tt}^2 \DT(u)(t,\xi) + |\xi|^2 \DT(u)(t,\xi) + \partial_t \DT(u)(t,\xi)  + m \DT(u)(t,\xi)=0 \\
\DT(u)(0,\xi) = \DT(u_0)(\xi)\\
\partial_t \DT(u)(0,\xi) = \DT(u_1)(\xi).
\end{cases}
\end{equation}
To simplify notation, let $U:=\DT(u)$, $U_0:=\DT(u_0)$ and $U_1:=\DT(u_1)$. In this notation, we have to solve the second order linear homogeneous ordinary differential equation in $t$
\begin{equation} \label{linearDT}
\partial_{tt}^2 U + b\partial_t U + (m+|\xi|^2) U=0,
\end{equation}
with initial conditions $U(0,\xi)=U_0(\xi)$ and $\partial_t U(0,\xi)=U_1(\xi)$. As is well known, the solution of this equation depends on the sign of
$$ D:= b^2-4(m+|\xi|^2)$$
and we treat the three different cases separately below.

\vspace{5pt}
\noindent\textbf{Case 1:} $D>0$. The solution of (\ref{linearDT}) is then
\begin{equation} \label{linearwavecase1}
U(t,\xi)
= e^{-\frac{b}{2}t}
	\left[
		U_0(\xi) \left(\cosh \frac{t\sqrt{D}}{2} +\frac{b}{\sqrt{D}} \sinh \frac{t\sqrt{D}}{2}\right)
		+U_1(\xi) \frac{2}{\sqrt{D}} \sinh \frac{t\sqrt{D}}{2}
	\right].
\end{equation}
This implies that
\begin{equation} \label{estimatewavecase1}
|U(t,\xi)|
\leq e^{-\frac{b}{2}t + \frac{\sqrt{D}}{2}t}
\left[ |U_0(\xi)| (1+ \frac{bt}{2}) + |U_1(\xi)| \frac{1}{\sqrt{D}}
\right],
\end{equation}
where we used the fact that $\frac{1-e^{-z}}{z}\leq 1$ for all $z\geq 0$.

For any $c>0$ we have $\frac{bct}{2} \leq e^{bct/2}$, so
$$ \frac{bt}{2} e^{-\frac{b}{2}t + \frac{\sqrt{D}}{2}t}
\leq \frac{1}{c} e^{-\frac{b}{2}(1-c)t + \frac{\sqrt{D}}{2}t}.$$
We know that for any $\xi\in\RR^N$ we have $D\leq b^2-4m$, so
$$ -\frac{b}{2}(1-c) + \frac{\sqrt{D}}{2} \leq -\frac{b}{2}(1-c)+\frac{\sqrt{b^2-4m}}{2} <0$$
if we choose $0<c<1-\frac{\sqrt{b^2-4m}}{b}$. In this case, we obtain $\delta>0$ independent of $\xi$ or $t$  such that
$$ \frac{bt}{2} e^{-\frac{b}{2}t + \frac{\sqrt{D}}{2}t}
\leq \frac{1}{c} e^{-\delta t}$$
and
$$e^{-\frac{b}{2}t + \frac{\sqrt{D}}{2}t}
\leq  e^{-\delta t}$$
for all $t$ and $\xi$. Consequently, (\ref{estimatewavecase1}) becomes
\begin{equation} \label{estimatecase1}
|U(t,\xi)| \leq C e^{-\delta t} \left[ |U_0(\xi)| + \frac{1}{\sqrt{D}} |U_1(\xi)| \right].
\end{equation}

\vspace{5pt}
\noindent\textbf{Case 2:} $D<0$. The solution of equation (\ref{linearDT}) in this case is
\begin{equation} \label{linearwavecase2}
U(t,\xi)
= e^{-\frac{bt}{2}} \left[ U_0(\xi) \left( \cos \frac{t\sqrt{|D|}}{2} + \frac{b}{\sqrt{|D|}} \sin \frac{t\sqrt{|D|}}{2} \right) + U_1(\xi) \frac{2}{\sqrt{|D|}} \sin \frac{t\sqrt{|D|}}{2} \right].
\end{equation}
We then have that
$$ |U(t,\xi)| \leq e^{-\frac{bt}{2}} \left[ |U_0(\xi)| (1+\frac{bt}{2} ) + \frac{2}{\sqrt{|D|}} |U_1(\xi)| \right],$$
and using, for example, the inequality $\frac{bt}{2}e^{-\frac{bt}{2}} \leq e^{-\frac{bt}{4}}$, we obtain the same estimate (\ref{estimatecase1}) as above for a constant $\delta>0$ independent of $t$ or $\xi$.

\vspace{5pt}
\noindent\textbf{Case 3:} $D=0$. The solution of (\ref{linearDT}) in this case is
\begin{equation} \label{linearwavecase3}
U(t,\xi) = e^{-\frac{bt}{2}} \left[ U_0(\xi) \left( 1+ \frac{bt}{2}\right) + tU_1(\xi) \right].
\end{equation}
A similar reasoning to above shows that in this case there exists a positive constant $\delta$ such that we have the estimate
\begin{equation} \label{estimatecase3}
|U(t,\xi)| \leq Ce^{-\delta t} \left[ |U_0(\xi)| + |U_1(\xi)| \right].
\end{equation}

\begin{prop} \label{estimateslinearDWE}
Fix $s\in\RR$. Let $u$ be a solution of the Cauchy problem (\ref{linearwave}) with initial data $u_0\in \DS{s}$ and $u_1\in \DS{s-1}$. Then there exists a constant $\delta>0$ (which does not depend on the initial data) such that, for all $t>0$ we have the estimates
\begin{equation} \label{linearestimate1}
\norm{u(t,\cdot)}_{\DS{s}}^2
\leq C e^{-2\delta t} \left( \norm{u_0}_{\DS{s}}^2 + \norm{u_1}^2_{\DS{s-1}}\right),
\end{equation}
Moreover, if $l\in \NN_0$ and $u_0 \in \DS{s+l}$ and $u_1\in \DS{s+l-1}$, then we have the estimate
\begin{equation} \label{linearestimate2} \norm{\partial_t^l u(t,\cdot)}_{\DS{s}}^2 \leq C e^{-2\delta t} \left( \norm{u_0}_{\DS{s+l}}^2 + \norm{u_1}^2_{\DS{s+l-1}}\right).
\end{equation}
\end{prop}

\begin{proof}
We first consider the case when $4m>b^2$. Then $D<0$ for all $\xi\in\RR^N$, so we are always in the Case 2 above. We have, keeping the same notation as above and using estimate (\ref{estimatecase1}), that
\begin{align*}
\norm{u(t,\cdot)}_{\DS{s}}^2
&= \int_{\RR^N} (1+|\xi|^2)^s |U(t,\xi)|^2 \diff\xi
\\
&\leq C e^{-2\delta t} \int_{\RR^N} (1+|\xi|^2)^s |U_0(t,\xi)|^2  \diff\xi
+ C e^{-2\delta t} \int_{\RR^N} \frac{(1+|\xi|^2)^s}{4m-b^2 + 4|\xi|^2} |U_1(t,\xi)|^2  \diff\xi.
\end{align*}
Due to the condition $4m>b^2$, the function $\xi\mapsto \frac{1+|\xi|^2}{4m-b^2+4|\xi|^2}$ is bounded on $\RR^N$, so the above estimates implies
\begin{align*}
\norm{u(t,\cdot)}_{\DS{s}}^2
&\leq C e^{-2\delta t} \int_{\RR^N} (1+|\xi|^2)^s |U_0(t,\xi)|^2  \diff\xi
+C e^{-2\delta t} \int_{\RR^N} (1+|\xi|^2)^{s-1} |U_1(t,\xi)|^2  \diff\xi
\\
&=C e^{-2\delta t} \left( \norm{u_0}_{\DS{s}}^2 + \norm{u_1}^2_{\DS{s-1}}\right),
\end{align*}
as required.

The remaining case $4m \leq b^2$ is a little more involved. The issue in the above computations is that now we have a singularity at $|\xi|=\frac{\sqrt{b^2-4m}}{2}$. We will now split our integral over $B:=\{ \xi\in\RR^N : |D|<1 \}$ and its complement, $B^C=\RR^N\setminus B$; more precisely, we use the following
\begin{equation} \label{decompositionB-BC}
\norm{u(t,\cdot)}_{\DS{s}}^2
= \int_{B} (1+|\xi|^2)^s |U(t,\xi)|^2 \diff\xi
+ \int_{B^C} (1+|\xi|^2)^s |U(t,\xi)|^2 \diff\xi.
\end{equation}

On $B^C$ we are always in Cases 1 and 2, and in both cases we have the same estimate (\ref{estimatecase1}). Using the same reasoning as above, because on $B^C$ the function $\xi\mapsto \frac{1+|\xi|^2}{|b^2-4m-4|\xi|^2|}$ is bounded, we deduce that
\begin{equation} \label{estimateBC}
\begin{aligned}
\int_{B^C} &(1+|\xi|^2)^s |U(t,\xi)|^2 \diff\xi
\\
&\qquad
\leq Ce^{-2\delta t} \left[\int_{B^C} (1+|\xi|^2)^s |U_0(\xi)|^2 \diff\xi
+ \int_{B^C} (1+|\xi|^2)^{s-1} |U_1(\xi)|^2 \diff\xi \right].
\end{aligned}
\end{equation}

We now turn to the integral over $B$. Starting from the explicit expressions we obtained in (\ref{linearwavecase1}) and (\ref{linearwavecase2}), using the fact that $\left| \frac{\sinh z}{z}\right| \leq 1$ and $\left| \frac{\sin z}{z}\right| \leq 1$ for all $z\neq 0$, we find the following estimate which holds in all cases
$$ |U(t,\xi)| \leq e^{-\tilde{\delta}t} \left[ |U_0(\xi)| (1+ \frac{bt}{2}) + t|U_1(\xi)| \right],$$
for some $\tilde{\delta}>0$. As in Case 3, this implies that the estimate (\ref{estimatecase3}) holds in general. Thus, we have the inequality
$$ \int_{B} (1+|\xi|^2)^s |U(t,\xi)|^2 \diff\xi
\leq Ce^{-2\delta t} \left[\int_B (1+|\xi|^2)^s |U_0(\xi)|^2 \diff\xi
+ \int_B (1+|\xi|^2)^s |U_1(\xi)|^2 \diff\xi \right].$$
But on $B$ we have $|D|<1$, so
$$4|\xi|^2 \leq |4|\xi|^2 + 4m - b^2| + b^2-4m <b^2-4m +1.$$
In other words, on $B$ the quantity $1+|\xi|^2$ is bounded above by a constant independent of $t$, so the estimate above becomes
\begin{equation} \label{estimateB}
\begin{aligned}
\int_{B} &(1+|\xi|^2)^s |U(t,\xi)|^2 \diff\xi
\\
&\qquad
\leq Ce^{-2\delta t} \left[\int_B (1+|\xi|^2)^s |U_0(\xi)|^2 \diff\xi
+ \int_B (1+|\xi|^2)^{s-1} |U_1(\xi)|^2 \diff\xi \right].
\end{aligned}
\end{equation}
Putting together (\ref{decompositionB-BC}), (\ref{estimateBC}) and (\ref{estimateB}), this completes the proof of estimate (\ref{linearestimate1}).

For the second estimate (\ref{linearestimate2}), we differentiate the explicit formulas (\ref{linearwavecase1}), (\ref{linearwavecase2}) and (\ref{linearwavecase3}) with respect to time, and then perform similar estimates as above. The proof of the Proposition is thus complete.
\end{proof}

\subsection{The Nonlinear Damped Wave Equation} \label{nonlinearwaveeq}

In this section we consider the nonlinear wave equation
\begin{equation} \label{dampedwaveeqn}
\begin{cases}
\partial_{tt}^2 u(t,x) - \Delta_k u(t,x) + b \partial_t u(t,x) + m u(t,x) = f(u)(t,x) \\
u(0,x)=\epsilon u_1(x) \\
\partial_t u(0,x) = \epsilon u_2(x),
\end{cases}
\end{equation}
where $b,m, \epsilon>0$, and the nonlinearity is a function $f:\RR\to\RR$ that will satisfy the assumptions
\begin{equation} \label{nonlinearassumption1}
f(0)=0
\end{equation}
and
\begin{equation} \label{nonlinearassumption2}
|f(a)-f(b)|\leq C (|a|^{p-1} + |b|^{p-1}) |a-b|,
\end{equation}
for some $p>1$.

Our main result is the following global existence result.

\begin{thm}
Suppose that $N+2\gamma>2$ and let $1\leq p \leq \frac{N+2\gamma}{N+2\gamma-2}$. Let $u_0\in \DSS$ and $u_1\in L^2(\mu_k)$. Assume that $f$ is a function that satisfies the assumption (\ref{nonlinearassumption1}-\ref{nonlinearassumption2}). Then there exists a constant $\epsilon_0>0$ such that for any $0<\epsilon<\epsilon_0$ the Cauchy problem (\ref{dampedwaveeqn}) has a unique global solution
$$u\in C(\RR_+;\DSS)\cap C^1(\RR_+;L^2(\mu_k)).$$
Moreover, there exists a constant $\delta_0>0$ such that the solution satisfies the estimate
$$ \norm{u(t,\cdot)}_{\DS{1}}+\norm{\partial u(t,\cdot)}_2 \leq C e^{-\delta_0 t}$$
for all $t>0$.
\end{thm}

\begin{proof}
\textbf{Part 1: Existence.} We will prove existence of the solution using the contraction theorem. Consider the Banach space
$$ X:=C(\RR_+;\DSS)\cap C^1(\RR_+;L^2(\mu_k)),$$
with the norm
$$ \norm{u}_X:= \sup_{t\geq 0} \left[ (1+t)^{-1/2} e^{\delta t} \left( \norm{u(t,\cdot)}_{\DS{1}} + \norm{\partial_t u(t,\cdot)}_2 \right)\right],$$
where $\delta>0$ is that obtained in Proposition \ref{estimateslinearDWE}. For $M>0$ which will be conveniently chosen later, consider
$$ Y:=\{ u\in X : \norm{u}_X \leq M\},$$
which is a closed subset of $X$ and thus a complete metric space in the metric induced from $X$.

Define the map
$$ S(u)(t,x):= \varphi(t,x) + \int_0^t T(f(u(s,\cdot)))(t-s,x) \diff s,$$
where $\varphi$ is the solution of the linear problem (the Cauchy problem (\ref{dampedwaveeqn}) with $f\equiv 0$), and $T(w)$ is defined generally, for $w:\RR^N\to\RR$, as the solution of the linear Cauchy problem
$$ \begin{cases}
\partial_{tt}^2 u - \Delta_k u + b\partial_t u + mu=0 \\
u(0,x)=0 \\
\partial_t u(0,x) = w(x).
\end{cases} $$
The strategy of our existence proof is to show that $S$ is a contraction on $Y$, i.e., to show that for all $u,v\in Y$ we have

\begin{description}
\item[(i)] $ \norm{S(u)}_X\leq M$ (in other words, $S$ maps $Y$ into itself), and

\item[(ii)] $ \norm{S(u)-S(v)}_X \leq c \norm{u-v}_X$ for some constant $0< c<1$.
\end{description}

\noindent Assuming these hold, then the Banach fixed point theorem guarantees the existence of a fixed point of $S$ in $Y$, which in turn is a solution to our Cauchy problem.

We will now focus on proving point ii., and i. will follow as a special case. To simplify the computations below, we introduce the notation
$$ \tilde{T}_u(t,x) := \int_0^t T(f(u(s,\cdot)))(t-s,x) \diff s.$$
We have, for any $t>0$,
\begin{equation} \label{waveproofestimate0}
\begin{aligned}
\norm{S(u)-S(v)(t,\cdot)}_{\DS{1}} +& \norm{\partial_t (S(u)-S(v))(t,\cdot)}_2
\\
&= \norm{(\tilde{T}_u-\tilde{T}_v)(t,\cdot)}_{\DS{1}} + \norm{\partial_t(\tilde{T}_u-\tilde{T}_v)(t,\cdot)}_2 .
\end{aligned}
\end{equation}

Noting that $T(f(u(s,\cdot)))-T(f(v(s,\cdot)))=T(f(u(s,\cdot))-f(v(s,\cdot)))$, we first estimate
\begin{align*}
\norm{(\tilde{T}_u-\tilde{T}_v)(t,\cdot)}_{\DS{1}}^2
&=\norm{\nabla_k (\tilde{T}_u-\tilde{T}_v)(t,\cdot)}_2^2
+ \norm{(\tilde{T}_u-\tilde{T}_v)(t,\cdot)}_2^2
\\
&=\int_{\RR^N} \left| \int_0^t \nabla_k\Big(T(f(u(s,\cdot))-f(v(s,\cdot)))\Big)(t-s,x) \diff s \right|^2 \diff \mu_k(x)
\\
&\qquad
+ \int_{\RR^N} \left|  \int_0^t T(f(u(s,\cdot))-f(v(s,\cdot)))(t-s,x) \diff s \right|^2 \diff \mu_k(x)
\\
&\leq t \int_{\RR^N} \int_0^t \left|\nabla_k\Big(T(f(u(s,\cdot))-f(v(s,\cdot)))\Big)(t-s,x)\right|^2 \diff s \diff \mu_k(x)
\\
&\qquad
+ t\int_{\RR^N} \int_0^t \left|T(f(u(s,\cdot))-f(v(s,\cdot)))(t-s,x)\right|^2 \diff s \diff \mu_k(x)
\\
&=t\int_0^t \norm{T(f(u(s,\cdot))-f(v(s,\cdot)))(t-s,\cdot)}_{\DS{1}}^2  \diff s.
\end{align*}
Here in the last line we simply interchanged the order of integration. Using the estimates from Proposition \ref{estimateslinearDWE}, we obtain
\begin{equation} \label{waveproofestimate1}
\norm{(\tilde{T}_u-\tilde{T}_v)(t,\cdot)}_{\DS{1}}^2
\leq Ct e^{-2\delta t} \int_0^t e^{2\delta s} \norm{f(u(s,\cdot))-f(v(s,\cdot))}_2^2 \diff s.
\end{equation}

Using Leibniz's formula for differentiation of an integral we have, since $T(w)(0,x)=0$ for any function $w$, that
$$ \frac{\partial}{\partial t} \left( \int_0^t T(f(w))(t-s,x) \diff s\right)
= \int_0^t \partial_t(T(f(w)))(t-s,x) \diff s.$$
Using this, we estimate
\begin{equation} \label{waveproofestimate2}
\begin{aligned}
\norm{\partial_t(\tilde{T}_u-\tilde{T}_v)(t,\cdot)}_2^2
&
= \int_{\RR^N} \left| \int_0^t \partial_t(T(f(u(s,\cdot))-f(v(s,\cdot))))(t-s,x) \diff s \right|^2 \diff\mu_k(x)
\\
&
\leq t\int_0^t \norm{\partial_t(T(f(u(s,\cdot))-f(v(s,\cdot))))(t-s,\cdot)}_2^2 \diff s
\\
&
\leq C t e^{-2\delta t} \int_0^t e^{2\delta s} \norm{f(u(s,\cdot))-f(v(s,\cdot))}_2^2 \diff s.
\end{aligned}
\end{equation}

We are left to estimate
$$ \norm{f(u(s,\cdot))-f(v(s,\cdot))}_2.$$
To this end, we use the assumption (\ref{nonlinearassumption2}) and H\"older's inequality with exponents $\frac{p}{p-1}$ and $p$ to obtain
\begin{align*}
\int_{\RR^N} & |f(u(s,x))-f(v(s,x))|^2 \diff\mu_k(x)
\\
&\qquad
\leq C \int_{\RR^N} \left[|u(s,x)|^{2(p-1)}+|v(s,x)|^{2(p-1)}\right] |u(s,x)-v(s,x)|^2 \diff\mu_k(x)
\\
&\qquad
\leq C \left[\norm{u(s,\cdot)}_{2p}^{2(p-1)} + \norm{v(s,\cdot)}_{2p}^{2(p-1)} \right] \norm{u(s,\cdot)-v(s,\cdot)}_{2p}^2.
\end{align*}

It is at this point that we use the Gagliardo-Nirenberg inequality in the form of Corollary \ref{GNcorollary2}. It is this result that imposes the restriction $1\leq p \leq \frac{N+2\gamma}{N+2\gamma-2}$. Thus, using this result, we have
\begin{equation} \label{proofwave}
\begin{aligned}
&\int_{\RR^N} |f(u(s,x))-f(v(s,x))|^2 \diff\mu_k(x)
\\
&\qquad\qquad\qquad
\leq C \norm{u(s,\cdot)-v(s,\cdot)}_{\DS{1}}^2
\left[ \norm{ u(s,\cdot)}_{\DS{1}}^{2(p-1)}
+\norm{v(s,\cdot)}_{\DS{1}}^{2(p-1)} \right].
\end{aligned}
\end{equation}

For the next step, we note that for any $w\in X$ we have
$$  \norm{w(s,\cdot)}_{\DS{1}}
\leq (1+s)^{1/2}e^{-\delta s} \norm{w}_X.$$
In light of this, and recalling that $\norm{u}_X\leq M$ and $\norm{v}_X\leq M$, (\ref{proofwave}) yields
$$\norm{f(u(s,\cdot))-f(v(s,\cdot))}_2^2
\leq C (1+s)^p e^{-2p\delta s} M^{2p-2} \norm{u-v}_X^2.$$
Replacing this in (\ref{waveproofestimate1}) and (\ref{waveproofestimate2}), we have obtained that
\begin{align*}
\norm{(\tilde{T}_u-\tilde{T}_v)(t,\cdot)}_{\DS{1}} &+ \norm{\partial_t(\tilde{T}_u-\tilde{T}_v)(t,\cdot)}_2
\\
&\qquad
\leq C \sqrt{t} e^{-\delta t} M^{p-1} \norm{u-v}_X
\left(\int_0^t (1+s)^p e^{-2(p-1)\delta s} \diff s\right)^{1/2}.
\end{align*}
The integral on the right hand side of this inequality can be bounded by a constant independently of $t$:
\begin{align*}
\int_0^t (1+s)^p e^{-2(p-1)\delta s} \diff s
&\leq \int_0^\infty (1+s)^p e^{-2(p-1)\delta s} \diff s
\\
&= \frac{e^{2(p-1)\delta}}{(2(p-1)\delta)^{p+1}} \int_{2(p-1)\delta}^\infty e^{-s} s^{p} \diff s <\infty.
\end{align*}
Consequently, going back to (\ref{waveproofestimate0}), we now have
\begin{align*}
(1+t)^{-1/2} e^{\delta t} &
	\left( \norm{S(u)-S(v)(t,\cdot)}_{\DS{1}} + \norm{\partial_t (S(u)-S(v))(t,\cdot)}_2
	\right)
\\
&\qquad\qquad\qquad\qquad\qquad\qquad\qquad
\leq C_1 \left(\frac{t}{1+t}\right)^{1/2} M^{p-1} \norm{u-v}_X,
\end{align*}
for a constant $C_1>0$. Taking supremum over all $t>0$, we finally obtain
\begin{equation}\label{waveestimateS(u)}
\norm{S(u)-S(v)}_X \leq C_1 M^{p-1} \norm{u-v}_X.
\end{equation}

We now want to bound $\norm{S(u)}_X$. On the one hand, putting $v=0$ in (\ref{waveestimateS(u)}), since $f(0)=0$ and so $S(0)=\varphi$, we obtain
$$ \norm{S(u)-\varphi}_X\leq C_1M^p.$$
On the other hand, from the our linear case estimates in Proposition \ref{estimateslinearDWE}, we obtain
$$ \norm{\varphi}_X
\leq C_2 \epsilon (\norm{u_0}_{\DS{1}}+\norm{u_1}_2),$$
for a constant $C_2>0$, and thus
\begin{equation} \label{waveestimateS(u)2}
 \norm{S(u)}_X \leq \norm{\varphi}_X + \norm{S(u)-\varphi}_X
 \leq C_1M^p + C_2 \epsilon (\norm{u_0}_{\DS{1}}+\norm{u_1}_2).
\end{equation}

Fixing $c\in (0,1)$, let $M=\left(\frac{c}{C_1}\right)^\frac{1}{p-1}$, so $C_1M^{p-1}=c$, and $\epsilon_0 = \frac{M(1-c)}{C_2(\norm{u_0}_{\DS{1}}+\norm{u_1}_2)}$. Then, for $0<\epsilon<\epsilon_0$, from (\ref{waveestimateS(u)}) and (\ref{waveestimateS(u)2}) we have
$$\norm{S(u)-S(v)}_X \leq c\norm{u-v}_X$$
and
$$ \norm{S(u)}_X \leq M,$$
i.e., points \textbf{(i)} and \textbf{(ii)} above. The proof of existence is thus complete.

\vspace{5pt}
\noindent\textbf{Part 2: Uniqueness.} Suppose that there exist $u,v\in X$ that both solve the Cauchy problem (\ref{dampedwaveeqn}). Let $\psi:=u-v$ and fix $t^*>0$. We will show that $\psi=0$ in $[0,t^*]\times\RR^N$ and since $t^*$ was arbitrary, this implies that $\psi\equiv 0$ and so $u=v$.

Both $u$ and $v$ are fixed points of the map $S$, and using the computations from the previous part we obtain that
\begin{align*}
\norm{\psi(t,\cdot)}_{\DS{1}}^2+\norm{\partial_t\psi(t,\cdot)}_{2}^2
&=\norm{\tilde{T}_u(t,\cdot)-\tilde{T}_v(t,\cdot)}_{\DS{1}}^2+\norm{\partial_t(\tilde{T}_u-\tilde{T}_v)(t,\cdot)}_2^2
\\
&\leq Cte^{-2\delta t} \int_0^t e^{2\delta s} \norm{f(u(s,\cdot))-f(v(s,\cdot))}_2^2 \diff s
\\
&\leq C t e^{-2\delta t} \int_0^t \left[\norm{u(s,\cdot)}_{\DS{1}}^{2(p-1)}+\norm{v(s,\cdot)}_{\DS{1}}^{2(p-1)} \right] \norm{\psi(s,\cdot)}_{\DS{1}}^2 \diff s.
\end{align*}
From the definition of the Banach space $X$, the map $t\mapsto \norm{u(t,\cdot)}_{\DS{1}}^{2(p-1)}+\norm{v(t,\cdot)}_{\DS{1}}^{2(p-1)}$ is continuous, and so it is bounded on the compact set $[0,t^*]$. Therefore, the inequality above becomes
\begin{align*}
\norm{\psi(t,\cdot)}_{\DS{1}}^2+\norm{\partial_t\psi(t,\cdot)}_{2}^2
&\leq C \int_0^t [\norm{\psi(s,\cdot)}_{\DS{1}}^2 + \norm{\partial_t\psi(s,\cdot)}_2^2] \diff s,
\end{align*}
and so, by Gronwall's lemma, we finally obtain
$$ \norm{\psi(t,\cdot)}_{\DS{1}}^2+\norm{\partial_t\psi(t,\cdot)}_{2}^2 =0$$
for all $0\leq t \leq t^*$. This proves uniqueness.

\textbf{Part 3: Estimates.} In the first two parts we have proved that our Cauchy problem has a unique solution, say $u\in Y$. By the definition of the $X$-norm, we have
$$ \norm{u(t,\cdot)}_{\DS{1}}+\norm{\partial u(t,\cdot)}_2
\leq (1+t)^{1/2} e^{-\delta t} M
\leq C e^{-\frac{\delta}{2}t}, $$
for any $t>0$.
\end{proof}

\noindent \textbf{Acknowledgements.} The authors wish to thank Sergey Tikhonov and Hatem Mejjaoli for pointing out missing references. The first author gratefully acknowledges financial support from EPSRC.

%\pagebreak
%\nocite{*}

\bibliographystyle{plain}
\bibliography{ref}

\end{document}